\documentclass[a4paper]{amsart}
\usepackage[active]{srcltx}
\usepackage[all]{xy}
\usepackage{subfig}
\usepackage{hyperref}
\usepackage{mathrsfs}

\usepackage{CJKutf8}
\usepackage{CJK}
\usepackage{amsfonts}
\usepackage{pifont}
\usepackage{bm}
\usepackage{latexsym,amsmath,amssymb,cite,amsthm}
\usepackage{color,eucal,enumerate,mathrsfs}
\usepackage[normalem]{ulem}
\usepackage{amsmath}

\numberwithin{equation}{section}
\theoremstyle{plain}
\newtheorem{theorem}{Theorem}[section]

\newtheorem{lemma}[theorem]{Lemma}
\newtheorem{proposition}[theorem]{Proposition}
\theoremstyle{definition}
\newtheorem{definition}[theorem]{Definition}
\theoremstyle{remark}
\newtheorem{assumption}[theorem]{Assumption}
\newtheorem{remark}[theorem]{Remark}

\newcommand{\opt}{\rm Opt}
\newcommand{\geo}{\rm Geo}

\newcommand{\ws}{\mathcal{W}_2}
\newcommand{\lmt}[2]{\mathop{\lim}_{{#1} \rightarrow {#2}} }

\newcommand{\lip}[1]{{\mathrm{lip}}({#1})}
\newcommand{\llip}{\mathrm{lip}}

\newcommand{\lmts}[2]{\mathop{\overline{\lim}}_{{#1} \rightarrow {#2}} }
\newcommand{\lmti}[2]{\mathop{\underline{\lim}}_{{#1} \rightarrow {#2}} }
\newcommand{\limstz}{\mathop{\overline{\lim}}_{t\downarrow 0} }
\newcommand{\limitz}{\mathop{\underline{\lim}}_{t\downarrow 0} }

\newcommand{\mm}{\mathfrak m}
\newcommand{\nn}{\mathfrak n}
\newcommand{\ms}{(X,\d,\mm)}
\newcommand{\cd}{{\rm CD}(K, \infty)}

\newcommand{\CD}{{\rm CD}}
\newcommand{\rcdkn}{{\rm RCD}^*(K, N)}
\newcommand{\rcd}{{\rm RCD}(K, \infty)}

\newcommand{\N}{\mathbb{N}}

\newcommand{\R}{\mathbb{R}}

\newcommand{\supp}{\mathop{\rm supp}\nolimits}   
\newcommand{\Lip}{\mathop{\rm Lip}\nolimits}
\renewcommand{\d}{{\mathrm d}}
\newcommand{\dt}{{\d t}}
\newcommand{\ddt}{{\frac \d\dt}}

\newcommand{\D}{{\mathrm D}}
\newcommand{\restr}[1]{\lower3pt\hbox{$|_{#1}$}}
\newcommand{\la}{{\langle}}
\newcommand{\ra}{{\rangle}}

\newcommand{\nchi}{{\raise.3ex\hbox{$\chi$}}}
\newcommand{\weakto}{\rightharpoonup}
\newcommand{\limi}{\varliminf}
\newcommand{\lims}{\varlimsup}

\title{{\bf Angles between curves in metric measure spaces}
}
\begin{document}
\author{ Bang-Xian Han} \thanks
{University of Bonn,
han@iam.uni-bonn.de} 
 \author {Andrea Mondino} \thanks
{University of Warwick,
A.Mondino@warwick.ac.uk}

\date{\today}
\maketitle

\begin{abstract}
The goal of the paper is to study the  angle between two curves in the  framework of   metric (and metric measure) spaces. More precisely, we  give  a new notion of angle between two curves in  a metric space. Such a notion  has a natural interplay with  optimal transportation  and is particularly  well suited for  metric measure spaces satisfying  the curvature-dimension condition. 
Indeed one of the  main results is the validity  of the  cosine formula on $\rcdkn$ metric measure spaces. As a consequence, the new introduced  notions are compatible with the corresponding  classical  ones for Riemannian manifolds, Ricci limit spaces and Alexandrov spaces.  
 \end{abstract}


\tableofcontents


\section{Introduction}

The  `angle'  between two curves is a basic concept of mathematics, which aims to quantify  the infinitesimal  distance between two crossing curves at a crossing point.
Such a notion is classical in Euclidean and in Riemannian geometries where a global (respectively infinitesimal) scalar product is given: the cosine of the  angle between two crossing curves is by definition  the scalar product of the velocity vectors.
If the space is not given an infinitesimal scalar product, it is a  challenging  problem to define angles in a sensible way.  
 In this paper, we will study this problem in a metric (measure) sense. More precisely,  consider  a  metric space  $(X,\d)$, a point   $p\in X$, and two geodesics  $\gamma, \eta$ such that $\gamma_0=\eta_{0}=p$. Our task is to  propose a meaningful definition of the angle between the curves $\gamma,\eta$ at the point $p$, denoted by  $\angle \gamma p \eta$, and to establish some interesting properties.  

We recall some examples first. Assume that  $\gamma$ and $\eta$ are geodesics, and the space $(X,\d)$ is an Alexandrov space, with upper  or lower  curvature bounds. From the monotonicity implied by the Alexandrov condition, it is  known (see for instance \cite{BBI-A}) that the angle  $\angle \gamma p\eta$  is well defined by the cosine formula:
\[
\angle \gamma p  \eta=\lmt{s,t}{0} \arccos \frac{s^2+t^2-\d^2(\gamma_{s}, \eta_{t})}{2st}=\lmt{t}{0} \arccos \frac{2t^2-\d^2(\gamma_t, \eta_{t})}{2t^2}.
\]
In order to define the angle for geodesics in a more general framework, a crucial observation is that a geodesic can be seen as  gradient flow of the distance function, i.e. a geodesic $\gamma$  `represents' the gradient of $-\d(\gamma_0, \gamma_{1}) \, \d(\gamma_{1},\cdot)$ on each point $\gamma_{t}$. Inspired by the seminal work of De Giorgi on gradient flows \cite{DeG}, given an arbitrary metric space $(X,\d)$ with a geodesic $\gamma:  [0,1]\to X$ and a Lipschitz function $f:X \to \R$, we say that $\gamma$ represents $\nabla f$ at time $0$, or $\gamma$ represents the gradient of the function $f$ at the point $p=\gamma(0)$ if the following inequality holds
\[
\lmti{t}{0} \frac {f(\gamma_{t})-f(\gamma_{0})}{t} \geq  \frac12 |\lip{ f}|^2  +\frac12 |\dot{\gamma}|^2,
\]
where $|\dot{\gamma}|=\d(\gamma_{0}, \gamma_{1})$ is the (constant, metric) speed of the geodesic $\gamma$.
Notice that the opposite inequality 
\[
\lmts{t}{0} \frac {f(\gamma_{t})-f(\gamma_{0})}{t} \leq  \frac12 |\lip{ f}|^2  +\frac12 |\dot{\gamma}|^2
\]
is always true by Leibniz rule and Cauchy-Schwartz inequality.  Hence $\gamma$ represents $\nabla f$ at time $0$ if and only if the equality holds. It is  easily seen that the geodesic $\gamma$ always represents the gradient of $f_{\gamma}(\cdot):=- \d(\gamma_{0}, \gamma_{1}) \, \d(\gamma_{1}, \cdot)$ at the point $\gamma_{0}$ (see for instance Lemma \ref{lem:fgamma}).
We then say  that  the angle $\angle \gamma p\eta$  between two geodesics $\gamma, \eta$ with $\gamma_{0}=\eta_{0}=p$ exists if the limit $\lim_{t \downarrow 0} \frac{f_{\gamma}(\eta_{t})- f_{\gamma}(\eta_{0})}{t}$ exists.  
In this case we set
\begin{equation}\label{eq:defanglef}
\angle \gamma p\eta:=\arccos \left( \frac{1}{|\dot{\gamma}| |\dot{\eta}|} \lim_{t \downarrow 0} \frac{f_{\gamma}(\eta_{t})- f_{\gamma} (\eta_{0})}{t}  \right).
\end{equation}
Notice that in case $(X,\d)$ is the metric space associated to a smooth Riemannian manifold $(M,g)$, the definition \eqref{eq:defanglef} reduces to  the familiar notion of angle 
$$\angle \gamma p\eta= \arccos g_{p}\left( \frac{\nabla f_{\gamma}(p)}{|\dot{\gamma}_{0}|},  \frac{\dot{\eta}_{0}}{|\dot{\eta}_{0}|} \right)= \arccos g_{p}\left( \frac{\dot{\gamma}_{0}}{|\dot{\gamma}_{0}|},  \frac{\dot{\eta}_{0}}{|\dot{\eta}_{0}|} \right).$$

Besides the case of Alexandrov spaces, a class of spaces where the angle is particularly well behaved is the one of Lipschitz-infinitesimally Hillbertian spaces. By definition, a metric measure space $(X,\d,\mm)$ is Lipschitz-infinitesimally Hillbertian if  for any pair of Lipschitz functions $f,g:X\to \R$ both  the limits for $\varepsilon \to 0$  of   $\frac{|\llip(f+\varepsilon g)|^{2}(x)- |\llip(f)|^{2}(x)}{2 \varepsilon}$ and    $\frac{|\llip(g+\varepsilon f)|^{2}(x)- |\llip(g)|^{2}(x)}{2 \varepsilon}$ exist and are equal for $\mm$-a.e. $x\in X$, where $\llip(f)$ is the local Lipschitz constant of $f$ (for the standard definition see \eqref{eq:DeflocLip}).
A remarkable example of  Lipschitz-infinitesimally Hillbertian spaces is given by the $\rcdkn$-spaces, a class of metric measure spaces satisfying Ricci curvature lower bound by $K\in \R$ and dimension upper bound by $N \in (1,\infty)$ in a synthetic sense such that the Laplacian is linear, and  which include as notable subclasses the Alexandrov spaces with curvature bounded below and the Ricci limit spaces   (i.e.  pointed measured Gromov-Hausdorff limits of sequences of  Riemannian manifolds with uniform lower Ricci curvature bounds).
\\In the class of Lipschitz-infinitesimally Hillbertian spaces, the second author \cite{M-A} introduced a notion of `angle between three points'; more precisely for every fixed pair of  points $p,q\in X$, for $\mm$-a.e. $x \in X$ the angle $\angle  pxq$ given by the formula
\begin{equation}\label{eq:angle3pointsIntro}
[0,\pi]\ni  \angle pxq:= \arccos \left( \lim_{\varepsilon\to 0} \frac{|\llip(r_{p}+\varepsilon r_{q})|^{2}(x)- |\llip(r_{p})|^{2}(x)}{2 \varepsilon}     \right),
\end{equation}
 is well defined, unique, and symmetric in $p$ and $q$. Here $r_{p}(\cdot):=\d(p,\cdot)$ is the distance function from $p$. A first result of the present paper is to relate the angle between three points with the angle between two geodesics: in Theorem \ref{thm:def1def2} we prove that if the angle $\angle pxq$ exists in the sense of \cite{M-A}  then also the angle between the geodesics $\gamma^{xp}, \gamma^{xq}$ joining $x$ to $p$ and $x$ to $q$ exists and coincides with the angle between the three points, i.e. $\angle \gamma^{xp}x\gamma^{xq}=\angle pxq$.  In particular it follows that  in a Lipschitz-infinitesimally Hilbertian geodesic space the angle between two geodesics in well defined in an a.e. sense.
 \\
 
 An important class of metric spaces are the spaces of probability measures  over  metric spaces endowed with the quadratic transportation distance: given a metric space $(X,\d)$ denote by  $\mathcal{W}_2:=(\mathcal{P}_2(X), W_2)$  the corresponding Wasserstein space. By using ideas similar to the ones above, together with Otto Calculus (see \cite{O-G}) and the calculus tools developed by Ambrosio-Gigli-Savar\'e \cite{AGS-C} and  Gigli \cite{G-O}, in Subsection \ref{SS:angleWasserstein}  we study in detail the angle between two geodesics  in $\mathcal{W}_2$. In particular if the underlying space $(X,\d,\mm)$ is an $\rcdkn$ space, we get the angle $\angle pxq$ between three points as the limit of the angle between the geodesics in $\mathcal{W}_2$ obtained by joining geodesically diffused approximations of Dirac masses centered at $p$, $x$ and $q$ (see Proposition \ref{prop:AngleWpoints} for the precise statement; see also Proposition \ref{prop:well-2} for a more detailed link with the optimal transport picture).
\\

Besides the case of Alexandrov spaces,  another class of spaces where the notion of angle is quite well understood is given by  Ricci limit spaces.  Indeed  it was proved by Honda  \cite{H-A} that if $(X,\d,\mm)$ is a Ricci-limit space, then  for $\mm$-a.e. $p \in X$  the angle  between two geodesics is well defined and  it satisfies the following one-variable cosine formula:
\begin{equation}\label{eq:angleCosLim}
\cos \angle \gamma p  \eta=\lmt{t}{0}\frac{2t^2-\d^2(\gamma_{t}, \eta_{t})}{2t^2}.
\end{equation}
One of the main goals of the present paper is to extend the validity of the formula \eqref{eq:angleCosLim} to metric measure spaces satisfying Ricci curvature lower bounds in a synthetic sense, the so-called $\rcdkn$-spaces (for the definition and basic properties of such spaces see Section \ref{sec:prel} and  references therein). This is  the content of the next theorem (corresponding to Theorem \ref{thm:cosine} in the body of the manuscript), which is one of the main results of the paper.

\begin{theorem}[Cosine formula for angles in $\rcdkn$ spaces]\label{thm:cosineIntro}
Let $\ms$ be an $\rcdkn$ space and fix $p,q \in X$. Then for  $\mm$-a.e. $x \in X$  there exist unique geodesics from $x$ to $p$ and from $x$ to $q$ denoted by  $\gamma^{xp},\gamma^{xq} \in \geo(X)$ and
\begin{equation}\label{eq:anglepxpgammaIntro}
  \angle \gamma^{xp} x \gamma^{xq}=  \angle p x q=   \lim_{t \to 0} \arccos \frac{2t^2-\d^2(\gamma^{xp}_{t}, \gamma^{xq}_{t})}{2t^{2}} ,\quad  \text{for  $\mm$-a.e. $x$}.
\end{equation}
\end{theorem}
The proof of Theorem \ref{thm:cosineIntro} in independent and different from the one of Honda \cite{H-A} for Ricci limit spaces: indeed Honda argues by getting estimates on the smooth approximating manifolds and then passes to the limit, while our proof for $\rcdkn$ spaces goes by arguing directly on the non smooth space $(X,\d,\mm)$. More precisely, we perform a blow up argument centered at $x$ and use that  for $\mm$-a.e. $x$ the tangent cone is unique and euclidean   \cite{GMR-E, MN-S}. From the technical point of view we  also make use of the  fine convergence  results for Sobolev functions proved in \cite{GMS-C, AH-N}, and we prove  estimates on harmonic approximations of distance functions (see in particular Proposition \ref{lemma-harmonicapprox}). Harmonic approximations of distance functions are well known  for smooth Riemannian manifolds with lower Ricci curvature bounds, and are indeed one of the key technical  tools in the   Cheeger-Colding theory of Ricci limit spaces \cite{CC-I,CC-II,CC-III}; on the other hand for non-smooth $\rcdkn$-spaces it seems they have not  yet appeared in the literature, and we expect them to be  a useful technical tool  in the future development of the field.
\\As a consequence of Theorem \ref{thm:cosineIntro}, we get that our definition of angle between two geodesics agrees (at least in a.e.  sense) with the Alexandrov's definition in case $(X,\d)$ is an Alexandrov space, and with the Honda's definition \cite{H-A} in case $(X,\d,\mm)$ is a Ricci limit space.

\noindent \textbf{Acknowledgement}: The first author would like to  thank Nicola Gigli, Shouhei Honda and  Karl Theodor Sturm for discussions on the topic.\\
%
\section{Preliminaries}\label{sec:prel}
\subsection{Metric measure spaces}
Let $(X,\d)$ be a  complete metric  space.  A continuous map $\gamma: [0,1] \mapsto X$ will be called \emph{curve}. The space of  curves  defined on $[0,1]$ with values in $X$ is denoted by $C([0,1],X)$.  The space $C([0,1],X)$ equipped with the uniform distance  is a  complete metric space.

We define the length of $\gamma$ by
\[
l[\gamma]:=\sup_\tau \mathop{\sum}_{i=1}^{n} \d(\gamma_{t_{i-1}},\gamma_{t_i})
\]
where $\tau:=\{0=t_0, t_1, ..., t_n=1\}$ is a partition of  $[0,1]$, and the $\sup$ is taken over all finite partitions. 
The space $(X, \d)$ is said to be a \emph{length space} if for any $x,y \in X$ we have
\[
\d(x,y) = \inf_\gamma l[\gamma],
\]
where the infimum is taken over all $\gamma \in C([0,1],X)$ connecting $x$ and $y$.
A geodesic from $x$ to $y$ is a curve $\gamma$ such that:
\[
\d(\gamma_s,\gamma_t) = |s - t|\d(\gamma_0,\gamma_1), \quad \forall t, s \in [0,1],~~~\gamma_0 = x,\gamma_1 = y.
\]
The space of all geodesics on $X$ will be denoted by $\geo(X)$. It is a closed subset of $C([0,1],X)$.

Given $p\in [1,+\infty]$ and a curve $\gamma$, we say that $\gamma$ belongs to $AC^p([0,1],X)$ if
\[
\d(\gamma_s,\gamma_t) \leq \int_s^t G(r) \,\d r, \quad \forall t,s \in [0,1], ~s<t \; ,
\]
for some $G \in L^p([0,1])$. In particular, the case $p=1$ corresponds to  absolutely continuous curves, whose class is denoted by $AC([0,1], X)$.
It is known that for $\gamma \in  AC([0,1], X)$, there exists an a.e. minimal function $G$ satisfying this inequality, called  metric derivative  and denoted by  $|\dot{\gamma}|$. The metric derivative  of $\gamma$ can be  computed for a.e. $t\in [0,1]$ as
\[
|\dot{\gamma}_t|:=\lmt{h}{0}\frac{\d(\gamma_{t+h},\gamma_t)}{|h|}.
\]
It is known that (see for example \cite{BBI-A}) the length of a curve $\gamma \in  AC([0,1], X)$ can be computed as
\[
l[\gamma]:=\int_0^1 |\dot{\gamma}_t|\,\dt.
\]
In particular, on a length space $X$ we have
\[
\d(x,y) = \inf_\gamma \int_0^1 |\dot{\gamma}_t|\,\dt
\]
where the infimum is taken among all $\gamma \in AC([0,1],X)$ which connect $x$ and $y$.

Given $f : X \mapsto \mathbb{R}$, the local Lipschitz constant $\lip{ f}: X \mapsto [0, \infty]$ is defined as
\begin{equation}\label{eq:DeflocLip}
\lip{f}(x):= \mathop{\overline{\lim}}_{y\rightarrow x}\frac{|f(y) -f(x)|}{\d(x, y)}
\end{equation}
if $x$ is not isolated, $0$ otherwise, while the  (global) Lipschitz constant  is defined as
\[
\Lip(f):= \mathop{\sup}_{x \neq y} \frac{|f(y)-f(x)|}{\d(x,y)}.
\]
If $(X, \d)$ is a length space, we have $\Lip(f)=\mathop{\sup}_{x}  \lip{f}(x)$.

We are not only  interested in metric structures, but also in the interaction between metric and measure. For the metric measure space $\ms$, basic assumptions used in this paper are:

\begin{assumption}\label{assumption}
The metric measure space $\ms$ satisfies:
\begin{itemize}
\item $(X,\d)$ is a complete and separable length space,
\item $\mm$ is a non-negative Borel  measure with respect to $\d$ and finite on bounded sets,
\item $\supp{\mm}=X$.
\end{itemize}
\end{assumption}

In this paper, we will often assume that the metric measure  space $\ms$  satisfies the  $\rcdkn$ condition, for some $K \in \R$ and $N \in [1, \infty]$ (when $N=\infty$ it is denoted by $\rcd$ ). The $\rcd$ and $\rcdkn$ conditions are refinements of the curvature-dimensions proposed by Lott-Sturm-Villani (see \cite{LV-R} and \cite{Sturm06I, Sturm06II} for $\cd$), and Bacher-Sturm (see \cite{BS-L} for ${\rm CD}^*(K, N)$) in order to isolate the non-smooth `Riemannian' structures from the `Finslerian' ones.  More precisely, the  ${\rm {RCD}}$ conditions are obtained by reinforcing the corresponding  ${\rm CD}$ conditions by adding the requirement that the Sobolev space  $W^{1,2}\ms$ is a Hilbert space (see the next subsection for more details).  It is then clear that the following relations hold
   \[
\rcdkn \subset {\rm CD}^*(K, N) ~~\text{and}~~ \rcd \subset \cd;
 \]
 moreover one has that
 \[
\rcdkn \subset \rcd ~~\text{and}~~{\rm CD}^*(K, N) \subset \cd.
 \]
 It is known that, for finite $N\in [1,\infty)$, a $ {\rm CD}^*(K, N)$ space $\ms$ satisfies  the following properties:
  \begin{itemize}
\item $\ms$ is  locally doubling  and therefore a locally compact space, \cite{BS-L};
\item $\ms$ supports a local Poincar\'e inequality, \cite{R-L}.
\end{itemize}
For more details about  $\rcd$ and $\rcdkn$ spaces, we refer  to \cite{AGS-M, AGMR-R, AMS-N, EKS-O}.

\subsection{Optimal transport and Sobolev functions}

The set of Borel probability measures on $(X,\d)$ will be denoted by $\mathcal{P}(X)$. We also use
$\mathcal{P}_2(X)\subseteq \mathcal{P}(X)$  to denote the set of measures with finite second moment, i.e. $\mu \in \mathcal{P}_2(X)$ if $\mu \in \mathcal{P}(X)$
and $\int \d^2(x,x_0)\, \d\mu(x) < +\infty$ for some (and thus every) $x_0 \in X$.
For $t\in [0,1]$, the evaluation map ${\rm e}_t: C([0,1],X) \rightarrow X$ is given  by
\[
{\rm e}_t(\gamma) := \gamma_t,\qquad\qquad\forall \gamma \in C([0,1],X).
\]
The space  $\mathcal{P}_2(X)$  is naturally endowed with the quadratic transportation distance ${W}_2$ defined by: 
\begin{equation}\label{eq:defW2}
{W}_2^2 (\mu,\nu) := \mathop{\inf}_{\pi} \int_{X\times X} \d^2(x,y)\, \d\pi(x,y),
\end{equation}
where the $\inf$ is taken among all couplings  $\pi \in \mathcal{P}(X\times X)$ with marginals $\mu$ and $\nu$, i.e. $(P_1)_\sharp \pi= \mu$ and  $(P_2)_\sharp \pi= \nu$ where $P_i$, $i=1,2$ are the projection maps onto the first and second coordinate respectively. The metric space $(\mathcal{P}_2(X),{W}_2)$ will be denoted by  $\mathcal{W}_2$. Let us recall that the infimum in the Kantorovich problem \eqref{eq:defW2} is always attained by an optimal coupling  $\pi$. We denote the set of optimal couplings between $\mu$ and $\nu$ by $\opt(\mu, \nu)$. Below we recall some fundamental properties of the metric space $\ws$ we will use throughout the paper.

\begin{proposition}[Geodesics in the Wasserstein space]

Let $(X,\d)$ be a  metric space and fix  $\mu_{0}, \mu_{1} \in \mathcal{P}_2(X)$. Then the curve $(\mu_t)_{t \in [0,1]} \subset \ws$
is a constant speed geodesic connecting $\mu_{0}$ and $\mu_{1}$, i.e. it satisfies
\begin{equation}
{W}_2(\mu_s,\mu_t)=|s-t| W_2(\mu_0,\mu_1), \quad \forall s,t \in [0,1]
\end{equation}
if and only if there exists  $\Pi \in  \mathcal{P}(\geo(X)) \subseteq \mathcal{P}(C([0,1],X)) $, called  optimal dynamical plan (or simply optimal plan), such that 
\[
\mu_t = ({{\rm e}_t})_\sharp \Pi \; \; \forall t \in [0,1]  \quad \text{and}\quad  ({{\rm e}_0}, {{\rm e}_1})_\sharp \Pi \in \opt(\mu_0, \mu_1).
\]
The set of optimal dynamical plans from $\mu_{0}$ to $\mu_{1}$ is denoted  by $\opt\geo(\mu_0, \mu_1)$.
\\Moreover, if $X$ is a  geodesic space, then  $\mathcal{W}_2$ is  also geodesic.
\end{proposition}

Absolutely continuous curves in $\mathcal{W}_2$  are characterized by the following theorem:
\begin{theorem}[Superposition principle, \cite{L-C}]\label{prop:superposition} Let $(X,\d)$ be a complete and separable metric space and let $(\mu_t) \in AC^2([0,1], \mathcal{P}_2(X))$. Then there exists a measure $\Pi \in \mathcal{P}(C([0,1],X))$ concentrated on
$AC^2([0,1],X)$ such that:
\begin{eqnarray*}
({\rm e}_t)_\sharp \Pi &=& \mu_t,~~~~~~\forall t \in [0,1]\\
\int |\dot{\gamma}_t|^2 \,\d \Pi(\gamma) &=&|\dot{\mu}_t|^2, \quad \text{for a.e}.~t\in [0,1].
\end{eqnarray*}

Moreover, the infimum of the energy $\int_0^1\int |\dot{\gamma}_t|^2 \,\d\Pi'(\gamma) \,\dt$ among all the  $\Pi' \in \mathcal{P}(C([0,1],X))$ satisfying $({\rm e}_t)_\sharp \Pi' = \mu_t$ for every $t\in[0,1]$ is attained  by such  $\Pi$.
\end{theorem}

\begin{definition}[Test plan] Let $(X,\d,\mm)$ be a metric measure space and $\Pi \in \mathcal{P}(C([0,1],X))$. We say  that $\Pi \in \mathcal{P}(C([0,1],X))$  has \emph{bounded compression} provided
there exists $C >0$ such that
\[
({\rm e}_t)_\sharp \Pi \leq C\mm,~~~\forall t \in [0,1].
\]
We say that $\Pi$ is a \emph{test plan} if it has bounded compression, is concentrated on $AC^2([0,1],X)$ and
\[
\int_0^1\int |\dot{\gamma}_t|^2 \,\d\Pi(\gamma)\, \dt < +\infty.
\]
\end{definition}

The notion of Sobolev function is given in duality with that of test plan:
\begin{definition}[The Sobolev class $S^{2}(X)$] Let $\ms$ be a metric measure space. A Borel function $f : X \rightarrow \mathbb{R}$ belongs to the
Sobolev class $S^2(X)$ (resp. $S^2_{loc}(X)$) provided there exists a non-negative function $G\in L^2(X,\mm)$ (resp. $L^2_{loc}(X,\mm)$) such that
\[
\int |f(\gamma_1)- f(\gamma_0)|\, \d\Pi(\gamma) \leq \int \int_0^1 G(\gamma_s)|\dot{\gamma}_s|\, \d s\, \d\Pi(\gamma),  \quad\forall \text{ test plan }\;\Pi.
\]
In this case, $G$ is called a 2-weak upper gradient of $f$, or simply  weak upper gradient.
\end{definition}
It is known, see e.g.\ \cite{AGS-C}, that there exists a minimal function $G$ in the $\mm$-a.e. sense among all the weak upper gradients of $f$. We denote such minimal function
by $|\D f|$ or $|\D f|_X$ to emphasize which  space we are considering and call it  \emph{minimal weak upper gradient}. Notice that if $f$ is Lipschitz, then $|\D f|\leq \lip f$ $\mm$-a.e., because $\lip f$ is a weak upper gradient of $f$.

It is known that the locality holds for $|\D f|$, i.e. $|\D f|=|\D g|$ $\mm$-a.e. on the set $\{ f=g\}$, moreover $S^2_{loc}\ms$ is a vector space and the inequality
\begin{equation}
\label{eq:sumd}
|\D(\alpha f+\beta g)|\leq |\alpha||\D f|+|\beta||\D g|,\quad\mm\text{-a.e.},
\end{equation}
holds for every $f,g\in S^2_{loc}\ms$ and $\alpha,\beta\in \R$. Moreover, the space $S^2_{loc}\cap L^\infty_{loc}\ms$ is an algebra, with the inequality
\begin{equation}
\label{eq:leibn}
|\D(fg)|\leq |f||\D g|+|g||\D f|, \quad\mm\text{-a.e.},
\end{equation}
being valid for any $f,g\in S^2_{loc}\cap L^\infty_{loc}\ms$.

The Sobolev space $W^{1,2}\ms$, also denoted by $W^{1,2}(X)$ for short, is defined as $$W^{1,2}(X):= S^2\ms  \cap L^2(X,\mm)$$ and is  endowed with the norm
\[
\|f\|^2_{W^{1,2}(X)}:=\|f\|^2_{L^2(X,\mm)}+\||\D f|\|^2_{L^2(X,\mm)}.
\]
$W^{1,2}(X)$ is always a Banach space, but in general it is not a Hilbert space.  $\ms$ is said infinitesimally Hilbertian  if $W^{1,2}(X)$ is a Hilbert space.

\bigskip

On an infinitesimally Hilbertian space, we have a natural pointwise inner product $\la\nabla \cdot, \nabla \cdot\ra: [W^{1,2}(X)]^2 \mapsto L^1(X)$ defined by
\[
\la \nabla f, \nabla g\ra:= \frac14 \Big{(}|\D (f+g)|^2-|\D (f-g)|^2\Big{)}.
\]

In order to prove the cosine formula we will use properties of harmonic functions in open sets of a m.m. space. Let us define the relevant quantities and recall the properties we will use; for simplicity, as always we assume the space $(X,\d)$ to be proper, complete and separable, and the measure $\mm$ to be finite on bounded sets (this indeed is the geometric case correspoding to $\rcdkn$ spaces, for $N<\infty$ we will be interested in). For the general case see for instance \cite{BB-N, G-O,GM-A}.

\begin{definition}[Sobolev  classes in $\Omega$ ]
Let $\ms$ be a m.m. space and let $\Omega\subset X$ be an open subset. The space $S^{2}(\Omega)$ is the space of Borel functions $f:\Omega \to \R$ such that $\chi f \in S^{2}(X)$ for any Lipschitz function $\chi : X \to [0,1]$ such that $\supp \chi \subset \Omega$, where $\chi f $ is taken  $0$ by definition on $X\setminus \Omega$. Let $W^{1,2}(\Omega):=L^{2}(\Omega)\cap S^{2}(\Omega)$  be the corresponding Sobolev space endowed with the natural norm, and  denote by $W^{1,2}_0(\Omega) \subset W^{1,2}(X)$ the closure of compactly supported Lipschitz functions on $\Omega$.
\end{definition}

\begin{definition}[Measure valued Laplacian]
Let $\ms$ be a m.m. space,  $\Omega\subset X$ an open subset and $f:\Omega \to \R$ a Borel function.  We say that $f$ is in the domain of the Laplacian in $\Omega$, and write $f \in {\rm D}({\bf \Delta}, \Omega)$  provided $f \in S^{2}(\Omega)$
and there exists a locally finite Borel measure $\mu$ on $\Omega$ such that for any $\varphi \in {\rm LIP}(X)$ with compact support contained in $\Omega$ it holds
$$\int_{X} \varphi \,\d \mu= -\int_{X} \la \nabla \varphi, \nabla f \ra    \, \d  \mm.  $$
In this case the measure $\mu$ is unique and we denote it by ${\bf \Delta} f\llcorner \Omega$, or simply  ${\bf \Delta} f$. If ${\bf \Delta} f \llcorner \Omega \ll \mm$, we denote its density  with respect to $\mm$ by $\Delta f \llcorner \Omega$ or simply by $\Delta f$.
\\ A function $f \in {\rm D}({\bf \Delta}, \Omega)$ is said to be \emph{harmonic in $\Omega$}, or simply harmonic, if   ${\bf \Delta} f \llcorner \Omega =0$.
\end{definition}

For simplicity we state the next proposition  for  $\rcdkn$ space, though it is valid more generally for doubling spaces supporting a weak-local 1-2 Poincar\'e inequality (see  \cite{BB-N} for details).

\begin{proposition}\label{prop-harmonic}
Let $\ms$ be a $\rcdkn$ space,  for some $K \in \R$ and $N\in [1,\infty)$, and let $\Omega\subset X$ be a bounded open set. Then the following properties hold.

\begin{itemize}
\item[i)] \emph{Regularity}. Let  $f:\Omega \to \R$ be harmonic in $\Omega$. Then $f$ admits a continuous representative (actually even locally Lipschitz).

\item [ii)] \emph{Comparison}. If $f,g \in {\rm D}({\bf \Delta},\Omega)$ are such that  $f \in W^{1,2}_{0}(\Omega),$  $|g|\leq C$ $\mm$-a.e. on $\Omega$  for some $C \in \R$ and ${\bf \Delta} (f+g) \geq 0$ then  $f\leq 2 C$ $\mm$-a.e. on $\Omega$.

\item [ii)] \emph{Existence and uniqueness of harmonic functions}. Assume that $\mm(X \setminus \Omega) >0$ and let $f\in W^{1,2}(X)$. Then there exists a unique harmonic function $g$ on $\Omega$ such that $f-g \in W^{1,2}_0(\Omega)$. 

\item [iv)] \emph{Strong maximum principle}. Let $f:\Omega \to \R$ be harmonic in $\Omega$ and assume that its continuous representative has a maximum at a point $x_0 \in \Omega$. Then $f$ is constant on the connected component of $\Omega$ containing $x_0$.


\end{itemize}
\end{proposition} 

In order to state the Laplacian Comparison Theorem, let us  introduce the coefficients $\tilde{\sigma}_{K,N}(\cdot):[0,\infty)\to\R$ defined by
 \[
\tilde{\sigma}_{K,N}(\theta):=\left\{
\begin{array}{ll}
\theta \sqrt{\frac{K}{N}} \, {\rm cotan} \left(\theta \sqrt{\frac{K}{N}} \right),&\qquad\textrm{ if }K>0,\\
1 &\qquad\textrm{ if }K=0 ,\\
\theta \sqrt{-\frac{K}{N}} \, {\rm cotanh} \left(\theta \sqrt{-\frac{K}{N}} \right),&\qquad\textrm{ if }K<0.
\end{array}
\right.
\]

\begin{theorem}[Laplacian comparison,  \cite{G-O}]\label{thm:LapComp}
Let $\ms$ be an $\rcdkn$ space for some $K\in \R$ and $N\in(1,\infty)$. Then
$$\frac{\d^2(x_{0}, \cdot)}{2}\in  {\rm D}({\bf \Delta}, X)  \quad \text{with} \quad {\bf \Delta} \frac{\d^2(x_0,\cdot)}{2} \leq N\, \tilde{\sigma}_{K,N}(\d (x_0, \cdot)) \,\mm \quad \forall x_0 \in X  $$ 
and
$$\d(x_0, \cdot) \in  {\rm D}({\bf \Delta}, X\setminus\{x_0\})  \quad \text{with} \quad {\bf \Delta} \d(x_0, \cdot) \llcorner {X\setminus \{x_0\}}\leq \frac{ N \, \tilde{\sigma}_{K,N}( \d(x_0, \cdot))-1}{\d(x_0, \cdot)}\, \mm \quad \forall x_0 \in X.   $$
\end{theorem}

\subsection{Pointed measured Gromov-Hausdorff convergence  and convergence of functions}
In order to study the convergence of possibly non-compact metric measure spaces, it is useful to fix  reference points. We then say that $(X,\d,\mm,\bar{x})$ is a pointed metric measure space, p.m.m.s. for short, if $(X,\d,\mm)$ is a m.m.s. as before and $\bar{x}\in X$ plays the role of reference point. Recall that, for simplicity, we always assume $\supp \mm=X$. We will adopt the following definition of convergence of p.m.m.s.  (see \cite{BBI-A}, \cite{GMS-C}
 and  \cite{V-O}):
\begin{definition}[Pointed measured Gromov-Hausdorff convergence]\label{def:conv}
A sequence $(X_j,\d_j,\mm_j,\bar{x}_j)$ is said to converge 
in the  pointed measured Gromov-Hausdorff topology (p-mGH for short) to 
$(X_\infty,\d_\infty,\mm_\infty,\bar{x}_\infty)$ if there 
exists a separable metric space $(Z,\d_Z)$ and isometric embeddings  
$\{\iota_j:(X_{j},\d_j)\to (Z,\d_Z)\}_{i \in \bar{\N}}$ such that
for every 
$\varepsilon>0$ and $R>0$ there exists $j_0$ such that for every $j>j_0$
\[
\iota_\infty(B^{X_\infty}_R(\bar{x}_\infty)) \subset B^Z_{\varepsilon}[\iota_j(B^{X_j}_{R+\varepsilon} (\bar{x}_j))]  \qquad \text{and} \qquad  \iota_j(B^{X_j}_R(\bar{x}_j)) \subset B^Z_{\varepsilon}[\iota_\infty(B^{X_\infty}_{R+\varepsilon} (\bar{x}_\infty))], 
\]
where $B^Z_\varepsilon[A]:=\{z \in Z: \, \d_Z(z,A)<\varepsilon\}$ for every subset $A \subset Z$, and 
\[
\lim_{j \to \infty} \int_Z \varphi \, \d ((\iota_j)_\sharp(\mm_j)) = \int_Z \varphi \, \d  ((\iota_\infty)_\sharp(\mm_\infty)) \qquad \forall \varphi \in C_b(Z), 
\]
where $C_b(Z)$ denotes the set of real valued bounded continuous functions with bounded support in $Z$.
\end{definition}
Sometimes in the following, for simplicity of notation, we will identify the spaces $X_j$ with 
their isomorphic copies $\iota_j(X_j)\subset Z$. It is obvious that this is in fact a notion of convergence for isomorphism classes of p.m.m.s., moreover it is induced by a metric   (see e.g. \cite{GMS-C} for details).
\\Next, following \cite{GMS-C}, we recall various notions of convergence of functions defined on p-mGH converging spaces.

\begin{definition}[Pointwise convergence of scalar valued functions]\label{def:pointConv}
Let $(X_j, \d_j, \mm_j, \bar{x}_j)$, $j\in \mathbb{N} \cup \{\infty\}$ be a p-mGH converging sequence of  p.m.m.s.  and let $f_j: X_j \mapsto \R, j\in \mathbb{N} \cup \{\infty\}$ be a sequence of functions. We say that $f_j $ converge pointwise to $ f_\infty$  provided:
\[
f_j(x_j) \to f_\infty(x_\infty)~~\text{ for every sequence of points} ~x_j \in X_j ~\text{such that}~ \iota_j(x_j) \to \iota_\infty(x_\infty) \text{ in } (Z,\d_{Z}).
\]
If for any $\epsilon>0$ there exists $N\in \mathbb N$ such that  $| f_j(x_j) - f_\infty(x_\infty)|\leq \epsilon$ for every $j \geq N$ and every $x_j \in X_j, x_\infty \in X_\infty$ with $\d_Z(\iota_j(x_j),\iota_\infty(x_\infty)) \leq \frac{1}{N}$, we say that $f_j \to f_\infty$ uniformly.
\end{definition}

\begin{definition}[$L^2$ weak and strong  convergence]\label{def:L2conv}
Let $(X_j, \d_j, \mm_j, \bar{x}_j)$, $j\in \mathbb{N} \cup \{\infty\}$ be a p-mGH converging sequence of pointed metric measure spaces and let $f_j\in L^2(X_j, \mm_j),  j\in \mathbb{N} \cup \{\infty\}$ be a sequence of functions.  
\begin{itemize}
\item  We say that $(f_j) $ converges \emph{weakly in $L^2$}  to $ f_\infty$  provided  $(\iota_j)_\sharp (f_j\,\mm_j ) \weakto (\iota_\infty)_\sharp (f_\infty\,\mm)$  weakly as Radon measures, i.e.
$$
\int_{X_{j}} f_{j}(x)\;  \varphi(\iota_{j}(x)) \; \d \mm_{j}(x) \to   \int_{X_{\infty}} f_{\infty}(x)\; \varphi(\iota_{\infty}(x)) \; \d \mm_{\infty}(x),  \qquad \forall \varphi \in C_b(Z), 
$$ 
and
$$\sup_{j \in \N} \int_{X_{j}} |f_{j}|^{2} \, \d\mm_{j}<\infty.$$

\item  We say that $(f_j) $ converges \emph{strongly in $L^2$}  to $ f_\infty$ provided it converges weakly in $L^{2}$ to $ f_\infty$ and moreover
$$
\lim_{j \to \infty} \int_{X_{j}} |f_{j}|^{2} \, \d\mm_{j}= \int_{X_{\infty}} |f_{\infty}|^{2} \, \d\mm_{\infty}.
$$
\end{itemize}
\end{definition}

\begin{definition}[$W^{1,2}$ weak and strong  convergence]
Let $(X_j, \d_j, \mm_j, \bar{x}_j)$, $j\in \mathbb{N} \cup \{\infty\}$ be a p-mGH converging sequence of pointed metric measure spaces and let $f_j\in W^{1,2}(X_j, \d_{j}, \mm_j),  j\in \mathbb{N} \cup \{\infty\}$ be a sequence of functions.  
We say that $(f_{j})$ converges \emph{weakly in $W^{1,2}$} to $f_{\infty}$ if $f_{j}$ are $L^{2}$-weakly convergent to $f$ and
 $$\sup_{j\in \N} \int_{X_{j}} |\D f_{j}|^{2} \, \d \mm_{j} < \infty. $$
 \emph{Strong convergence in $W^{1,2}$} is defined by requiring $L^{2}$-strong convergence 
of the functions and that 
$$
\lim_{j \to \infty} \int_{X_{j}}  |\D f_{j}|^{2} \, \d \mm_{j} =   \int_{X_{\infty}}  |\D f_{\infty}|^{2} \, \d \mm_{\infty}.
$$ 
\end{definition}

The next result proved in \cite[Corollary 5.5]{AH-N} (see also \cite[Corollary 6.10]{GMS-C}) will be useful in the sequel.

\begin{proposition} \label{prop-1order-converge}
Let $(X_j, \d_j, \mm_j, \bar{x}_j)$, $j\in \mathbb{N} \cup \{\infty\}$ be a p-mGH converging sequence of  pointed metric measure spaces.
 If for every $j \in \N$ one has $f_j \in W^{1,2}(X_i)$, $f_j \in {\rm D}({\bf \Delta}_j, X_{j})$ with $\Delta_{j} f_j$  uniformly bounded in  $L^2$, and $(f_{j})$ converges strongly in $L^{2}$ to $f_{\infty}$, then $f_{\infty}\in {\rm D}({\bf\Delta}_{\infty}, X_{\infty})$ and $(f_j)$ converges  to
$f_{\infty}$
 strongly in $W^{1,2}$.
 \end{proposition} 

\subsection{Euclidean tangent cones to $\rcdkn$ spaces}
Let us first recall the notion of measured tangents. Let  $\ms$ be a m.m.s.,  $\bar x\in X$ and $r\in(0,1)$; we consider the rescaled and normalized p.m.m.s. $(X,r^{-1}\d,\mm^{\bar{x}}_r,\bar x)$ where the measure $\mm^{\bar x}_r$ is given by
\begin{equation}
\label{eq:normalization}
\mm^{\bar x}_r:=\left(\int_{B_r(\bar x)}1-\frac 1r\d(\cdot,\bar x)\,\d\mm\right)^{-1}\mm.
\end{equation}
Then we define:
\begin{definition}[Tangent cone and regularity]
Let  $(X,\d,\mm)$ be a m.m.s. and  $\bar x\in X$. A p.m.m.s.  $(Y,\d_Y,\nn,y)$ is called a
\emph{tangent} to $(X,\d,\mm)$ at $\bar{x} \in X$ if there exists a sequence of rescalings $r_j \downarrow 0$ so that
$(X,r_j^{-1}\d,\mm^{\bar{x}}_{r_j},\bar{x}) \to (Y,\d_Y,\nn,y)$ as 
$j \to \infty$ in the p-mGH sense.
We denote the collection of all the tangents of $(X,\d,\mm)$ at 
$\bar{x} \in X$ by ${\rm Tan}(X,\d,\mm,\bar{x})$.  A point $\bar{x} \in X$ is called \emph{regular} if the tangent is unique and euclidean, i.e. if  ${\rm Tan}(X,\d,\mm,\bar{x})=\{(\R^{n}, \d_E, {\mathcal L}_{n}, 0^{n})\}$, where $\d_{E}$ is the Euclidean distance and ${\mathcal L}_{n}$ is the properly rescaled Lebesgue measure of $\R^{n}$.
\end{definition}

 The a.e.  regularity was settled for Ricci-limit spaces by Cheeger-Colding  \cite{CC-I, CC-II, CC-III}; for an  $\rcdkn$-space $(X,\d,\mm)$,  it was proved  in \cite{GMR-E} that for $\mm$-a.e. $x\in X$ there exists a blow-up sequence converging to a Euclidean space. The $\mm$-a.e.  uniqueness of the blow-up limit, together with the rectifiability of an $\rcdkn$-space, was then established in \cite{MN-S}.  More precisely the following  holds:
 
\begin{theorem}[$\mm$-a.e. infinitesimal regularity of $\rcdkn$-spaces]  \label{RCD-reg}
Let $(X,\d,\mm)$ be an $\rcdkn$-space for some $K\in \R, N \in (1,\infty)$. Then $\mm$-a.e. $x\in X$ is a regular  point, i.e. for $\mm$-a.e.  $x \in X$ there exists $n=n(x)\in [1,N]\cap \N$ such that, for any sequence $r_{j}\downarrow 0$,  the rescaled  pointed metric measure spaces $(X, r_{j}^{-1} \d, \mm^{x}_{r_{j}}, x)$ converge in the p-mGH sense to the pointed Euclidean space $(\R^{n}, \d_E, {\mathcal L}_{n},  0^{n})$. 
\end{theorem}

\section{Definition of angle}

\subsection{Angle between three points}

In \cite{M-A}, the second author proposed a notion of angle between three points $p,x,q \in X$ in a metric space $(X,\d)$. In general such an angle  is non unique, the possible causes of non-uniqueness being a lack of regularity of the distance function (e.g. $x$ is in the cut locus of $p$ or $q$) or a lack of  infinitesimal  strict convexity  of the distance function (for more details we refer to \cite[Sections 1,2]{M-A}). For simplicity, here we only treat the case when the angle is unique.
Given two points  $p,q \in X$, consider the distance functions 
\begin{equation} \label{eq:rprq}
r_p(\cdot):=\d(p,\cdot),\quad r_q(\cdot):=\d(q,\cdot).
\end{equation}

\begin{definition}\label{def:anglepxq}
We say that the angle $\angle pxq$ exists if and only if the limit for $\varepsilon\to 0$ of the quantity  $\frac{|\llip(r_{p}+\varepsilon r_{q})|^{2}(x)- |\llip(r_{p})|^{2}(x)}{2 \varepsilon}$ exists. In this case we set
\begin{equation} \label{eq:rprq}
[0,\pi]\ni  \angle pxq:= \arccos \left( \lim_{\varepsilon\to 0} \frac{|\llip(r_{p}+\varepsilon r_{q})|^{2}(x)- |\llip(r_{p})|^{2}(x)}{2 \varepsilon}     \right).
\end{equation}
\end{definition}
Note that if $(X,\d)$ is a smooth Riemannian manifold and $x$ is not in the cut locus of $p$ and $q$, then  $\angle pxq$ is the angle based at $x$ between $\nabla r_{p}(x)$ and $\nabla r_{q}(x)$; in other words $\angle pxq$ is the angle based at $x$ ``in direction of $p$ and $q$''. 
As already mentioned, for a general triple $pxq$  in a general metric space $(X,\d)$ the angle $\angle pxq$ may not exist; moreover, even if both  $\angle pxq$ and  $\angle qxp$ exist they may not be equal in general. On the other hand, such a definition satisfies some natural properties one expects from the geometric picture: the angle is invariant under  a constant rescaling of the metric $\d$, moreover for any two points $x,p\in X$ the angle  $\angle pxp$ always exists and,  if $(X,\d)$ is a length space,   is equal to $0$.

We now discuss an important class  of metric measure spaces $\ms$ where the angle exists and is symmetric in an a.e. sense, the so called Lipschitz-infinitesimally Hilbertian spaces.

\begin{definition}
A metric measure space $\ms$ is said to be  Lipschitz-infinitesimally Hilbertian if for any pair of Lipschitz functions $f,g\in \mathrm{LIP}(X)$ both  the limits for $\varepsilon \to 0$  of   $\frac{(|\llip(f+\varepsilon g)|^{2}(x)- |\llip(f)|^{2}(x)}{2 \varepsilon}$ and    $\frac{|\llip(g+\varepsilon f)|^{2}(x)- |\llip(g)|^{2}(x)}{2 \varepsilon}$ exist and are equal for $\mm$-a.e. $x\in X$, i.e.
\begin{equation}\label{eq:defLipInfHilb}
\lim_{\varepsilon\to 0}\frac{|\llip(f+\varepsilon g)|^{2}(x)- |\llip(f)|^{2}(x)}{2 \varepsilon}= \lim_{\varepsilon \to 0}\frac{|\llip(g+\varepsilon f)|^{2}(x)- |\llip(g)|^{2}(x)}{2 \varepsilon}, \quad \mm \text{-a.e. } x.
\end{equation}
 \end{definition}
 
 It is clear that if $\ms$ is  Lipschitz-infinitesimally Hilbertian then, given $p,q\in X$, for $\mm$-a.e. $x \in X$ both the angles  $ \angle pxq,   \angle qxp$ exist and $ \angle pxq= \angle qxp$.

 \begin{remark}\label{rem:LIH}
 The concept of Lipschitz-infinitesimally Hilbertian space was proposed in \cite{M-A} as a variant of the notion of infinitesimally Hilbertian space introduced in  \cite{AGS-M, G-O}, using the  language  of minimal weak upper gradients;  let us mention that  Lipschitz-infinitesimally Hilbertian always implies infinitesimally Hilbertian, but the converse is not clear in general. An important class of spaces where also  the converse implication holds is the one of locally doubling spaces satisfying a weak Poincar\'e inequality.  Indeed, by a celebrated result of Cheeger \cite{C-D},  we have that for every $f \in  \mathrm{LIP}(X)$ it holds $\lip f = |\D f|$ $\mm$-a.e., in other words the local  Lipschitz constant is equal to the minimal weak upper gradient $\mm$-a.e.  In particular for $\CD^{*}(K,N)$ spaces, $K\in \R, N \in [1, \infty)$ the two notions are equivalent. 
 For more details we refer to \cite[Remark 3.3]{M-A}.
 \\It follows that $\rcdkn$-spaces are  Lipschitz-infinitesimally Hilbertian, for $N<\infty$; let us recall that the class of  $\rcdkn$-spaces include finite dimensional Alexandrov spaces with curvature bounded below  and Ricci limit spaces as remarkable sub-classes.
 \end{remark}

\subsection{Angle between two geodesics}\label{sec:AngleGeod}

First of all observe that if $(X,\d)$ is a metric space and $\gamma\in \geo(X)$ is a geodesic,  then  $|\dot{\gamma}_{t}|=\d(\gamma_{0}, \gamma_{1})$ for a.e. $t \in [0,1]$; we will denote such a constant simply by $|\dot{\gamma}|$.  The next definition is inspired by the De Giorgi's metric
concept of gradient flow \cite{DeG}.
\begin{definition}[A geodesic representing the gradient of a Lipschitz function]\label{def:gradfgamma}
Let $f \in \mathrm{LIP} (X)$ be a Lipschitz function on $(X,\d)$. We say that $\gamma\in \geo(X)$ represents $\nabla f$ at time $0$, or $\gamma \in \geo(X)$ represents the gradient of  $f$ at the point $x=\gamma_{0}$ if the following inequality holds
\begin{equation}\label{eq:defgradfgamma}
\lmti{t}{0} \frac {f(\gamma_{t})-f(\gamma_{0})}{t} \geq  \frac12 \lip{ f}^{2} (\gamma_{0})  +\frac12 |\dot{\gamma}|^2.
\end{equation}
\end{definition}
Notice that the opposite inequality is always true, indeed 
\begin{align*}
\lmts{t}{0} \frac {f(\gamma_{t})-f(\gamma_{0})}{t} \leq   \lip{ f}(\gamma_{0})  \;  |\dot{\gamma}| \leq   \frac12 \lip{ f}^{2} (\gamma_{0}) +\frac12 |\dot{\gamma}|^2.
\end{align*}
 Hence $\gamma \in \geo(X)$ represents $\nabla f$ at time $0$ if and only if the equality holds. Note that, in the case of Riemannian manifolds, $\gamma$ represents $\nabla f$ at time $0$ if and only if $\dot{\gamma}_0=\nabla f$. 
\\ It is easy to check that the geodesic $\gamma \in \geo(X)$ represents the gradient of $f \in  \mathrm{LIP} (X)$ at $x \in X$ if and only if for every $\alpha\in (0,1)$ the rescaled geodesic $\tilde{\gamma}\in \geo(X)$ defined by $\tilde{\gamma}_{t}:=\gamma_{\alpha t}$, $\forall t \in [0,1]$, represents the gradient of the Lipschitz function $\alpha f$ at $x$. 
 In the next lemma we give a simple but  important example of a geodesic representing the gradient of a function.
 
 \begin{lemma}\label{lem:fgamma}
 Let $(X,\d)$ be a metric space, fix $p \in X$ and let $r_{p}(\cdot):=\d(p,\cdot)$. If for some $x \in X$ there exists a geodesic $\gamma^{xp}\in \geo(X)$ such that $\gamma_{0}=x$ and $\gamma_{1}=p$ then $\gamma^{xp}$ represents the gradient of $f(\cdot):=- \d(p,x)\,  r_{p}(\cdot) $ at $x$.
 \end{lemma}

\begin{proof}
For every $t\in (0,1)$ it holds
\begin{align*}
\frac {f(\gamma^{xp}_{t})-f(\gamma^{xp}_{0})}{t}&= \d(p,x)\; \frac{\d(p,x)-\d(p, \gamma^{xp}_{t})}{t} =  \d(p,x) \; \frac{\d(x, \gamma^{xp}_{t})}{t}=  \d(p,x) \; \frac{t  \d(x,p)}{t}\\
&=  \d(p,x)^{2}.
\end{align*}
On the other hand, by triangle inequality it is clear that $\llip( r_{p})\leq 1$ and with  an analogous argument as above it is  easily checked that actually $\llip( r_{p}) (x)=1$.  Therefore  $\llip(f)(x)= \d(p,x)=: |\dot{\gamma}^{xp}|$  and the claim follows.
\end{proof} 

We can now define the angle between two geodesics.

\begin{definition}[Angle between two geodesics]\label{def:angle-1}
Let $(X,\d)$ be a metric space and let $\gamma, \eta\in \geo(X)$ be two geodesics with $\gamma_{0}=\eta_{0}=p$. Let $f \in LIP(X)$ be a Lipschitz function such that $\gamma$ represents the gradient of $f$ at time $0$.  We say that the angle  $\angle \eta p \gamma$ exists if and only if  the limit  as $t\downarrow 0$ of  $\frac{f(\eta_t)-f(\eta_0)}{t}$ exists. In this case we set
 \begin{equation}\label{eq:defanglegeod}
[0,\pi] \ni \angle \eta p \gamma:=\arccos \left(  \frac{ 1} { |\dot{\eta}||\dot{\gamma}|}  \mathop{\lim}_{t\downarrow 0} \frac{f (\eta_t)-f(\eta_0)}{t}  \right) .
  \end{equation}
\end{definition}

\begin{remark}[Locality of the angle between two geodesics]\label{rem:restrAngleGeo}
It is easily seen that the angle between the two geodesics $\gamma, \eta\in \geo(X)$ at the point $p=\gamma_{0}=\eta_{0}$ depend just on the germs of the curves at $p$. To see that, fix arbitrary $T_{\gamma}, T_{\eta}\in (0,1)$ and call  $\tilde{\gamma}, \tilde{\eta}$ the restrictions of $\gamma, \eta$ to $[0,T_{\gamma}], [0,T_{\eta}]$ properly rescaled, i.e:
$$\tilde{\gamma}(t):=\gamma(T_{\gamma} t), \quad  \tilde{\eta}(t):=\eta(T_{\eta} t), \quad  \forall t \in [0,1].$$
Of course we still have $\tilde{\gamma}, \tilde{\eta}\in \geo(X)$, and it is readily seen that $\tilde{\gamma}$ represents the gradient of $\tilde{f}:=T_{\gamma} f$. It follows that  $\angle \eta p \gamma$ exists if and only if $\angle \tilde{\eta} p \tilde{\gamma}$ exists, and in this case it holds
\begin{align*}
 \angle \eta p \gamma&:=\arccos \left(  \frac{ 1} { |\dot{\eta}||\dot{\gamma}|}  \mathop{\lim}_{t\to 0} \frac{f(\eta_t)-f(\eta_0)}{t}  \right)= \arccos \left(  \frac{ 1} { |\dot{\tilde{\eta}}||\dot{\tilde{\gamma}}|}  \mathop{\lim}_{t\to 0} \frac{\tilde{f}(\tilde{\eta}_t)-\tilde{f}(\tilde{\eta}_0)}{t}  \right) \\
 &= \angle \tilde{\eta} p \tilde{\gamma}.
\end{align*}
\end{remark}

\begin{remark}[Dependence on the function $f$]\footnote{AM:added this remark}
Note also in the generality of metric spaces, the angle $\angle \gamma p\eta$ as given  in Definition \ref{def:angle-1}  may depend on the function $f$ chosen in \eqref{eq:defanglegeod} (for instance this is the case of a tree with a vertex in $p$ and two edges made by $\gamma$ and $\eta$). In case $(X,\d,\mm)$ is an $\rcdkn$-space we will see later in the paper that actually the angle between two geodesics is well defined for $\mm$-a.e. base point $p$ just in terms of the geometric data, so it does not depend on the choice of $f$.
In the general case of a metric space, a way to overcome the problem would be to fix a canonical Lipschitz function $f$ such that $\gamma$ represents $\nabla f$ at time $0$.  In view of Lemma \ref{lem:fgamma}, a natural choice is to consider $f_{\gamma}(\cdot):=-\d(\gamma_{0}, \gamma_{1}) \d(\gamma_{1}, \cdot)$. In case $(X,\d,\mm)$ is not an $\rcdkn$ space we will tacitly make such a choice so to have a good definition. 
\end{remark}

The next goal is to relate the angle between three points with the angle between two geodesics, i.e. relate  Definitions \ref{def:anglepxq} and   \ref{def:angle-1}.

\begin{theorem}\label{thm:def1def2} \footnote{AM: this thm and the proof  is improved from the last version, where there was problem due to the lack of locality of the angle between three points}
Let $(X,\d)$ be a metric space and let $p,x,q \in X$ satisfy the following assumptions:
\begin{itemize}
\item  the angle $\angle pxq$ exists in the sense of Definition \ref{def:anglepxq},
\item there exists geodesic $\gamma^{xp}, \gamma^{xq}\in \geo(X)$ from $x$ to $p$ and from $x$ to $q$ respectively. 
\end{itemize}
Then the angle $ \angle \gamma^{xp} x \gamma^{xq}$ exists in the sense of Definition \ref{def:angle-1} and
\begin{equation}\label{eq:anggammapxq}
\angle \gamma^{xp} x \gamma^{xq} = \angle pxq.
\end{equation}
\end{theorem}
 Note that if  $(X,\d)$ is a geodesic Lipschitz-infinitesimally Hilbertian space then for every given $p,q\in X$ the two assumptions of Theorem \ref{thm:def1def2} are satisfied for $\mm$-a.e. $x\in X$. This is in particular the case for $\rcdkn$ spaces (see Remark \ref{rem:LIH}).
 
 \begin{proof}
  Without loss of generality we can assume that $\d(p,x) \leq \d(q,x)$, otherwise just exchange the role of $p$ and $q$ in the arguments below.
 Let $f_{p}(\cdot):=- \d(p,x)\,  r_{p}(\cdot)$ and $f_{q}(\cdot):=- \d(p,x)\,  r_{q}(\cdot)$.  Recall from Lemma \ref{lem:fgamma} that $\gamma^{xp}$ represents $\nabla f_{p}$ at $x=\gamma^{xp}_{0}$ in the sense of Definition \ref{def:gradfgamma}, i.e.
 \begin{equation}\label{eq:pfgammaf1}
 \limi_{t\downarrow0} \frac{ f_{p}(\gamma^{xp}_{t})- f_{p}(\gamma^{xp}_{0})}{t}\geq \frac{\llip (f_{p})^{2}(x)}{2} + \frac{|\dot{\gamma}^{xp}|^{2}}{2}.
 \end{equation}
 On the other hand
 \begin{equation}\label{eq:pfgammaf2}
 \lims_{t\downarrow0} \frac{ (f_{p}+\varepsilon f_{q}) (\gamma^{xp}_{t})-  (f_{p}+\varepsilon  f_{q}) (\gamma^{xp}_{0}) }{t} \leq \frac{\llip( f_{p}+\varepsilon  f_{q})^{2}(x)}{2}+  \frac{|\dot{\gamma}^{xp}|^{2}}{2}.
 \end{equation}
 Subtracting \eqref{eq:pfgammaf1} from \eqref{eq:pfgammaf2} yields
 \begin{align}
 \lims_{t\downarrow0} \, \varepsilon  \;  \frac{ f_{q} (\gamma^{xp}_{t})-   f_{q} (\gamma^{xp}_{0}) }{t} &\leq \frac{\llip(f_{p}+\varepsilon  f_{q})^{2}(x)- \llip (f_{p})^{2}(x)}{2} \nonumber \\
 &=   \d(p,x)^{2} \;  \frac{\llip(r_{p}+ \varepsilon  r_{q})^{2}(x)- \llip (r_{p})^{2}(x)}{2}. \label{eq:pfgammaf3}
  \end{align}
If $\varepsilon>0$,   dividing both sides by $\varepsilon   \d(p,x)^{2}$ and  letting $\varepsilon \downarrow 0$ we get
\begin{align*}
 \lims_{t\downarrow0} \, \frac{1}{ \d(p,x)^{2}} \frac{ f_{q} (\gamma^{xp}_{t})-   f_{q} (\gamma^{xp}_{0}) }{t} &\leq \lim_{\varepsilon \downarrow 0}  \frac{\llip(r_{p}+ \varepsilon  r_{q})^{2}(x)- \llip (r_{p})^{2}(x)}{2 \varepsilon}= \cos(\angle pxq),
  \end{align*}
  where in the last identity we used the assumption that $\angle pxq$ exists.
  Analogously, if $\varepsilon<0$,   dividing both sides by $\varepsilon   \d(p,x)^{2}$ and  letting $\varepsilon \uparrow 0$ we get
  \begin{align*}
 \limi_{t\downarrow0} \, \frac{1}{ \d(p,x)^{2}} \frac{ f_{q} (\gamma^{xp}_{t})-   f_{q} (\gamma^{xp}_{0}) }{t} &\geq \lim_{\varepsilon \uparrow 0}  \frac{\llip(r_{p}+ \varepsilon  r_{q})^{2}(x)- \llip (r_{p})^{2}(x)}{2 \varepsilon}= \cos(\angle pxq).
  \end{align*}
  The combination of the last two inequalities gives the existence of the limit for $t\downarrow 0$ of  $\frac{1}{ \d(p,x)^{2}} \frac{ f_{q} (\gamma^{xp}_{t})-   f_{q} (\gamma^{xp}_{0}) }{t}$ and, more precisely, 
  \begin{align*}
 \lim_{t\downarrow 0} \, \frac{1}{ \d(p,x)^{2}} \frac{ f_{q} (\gamma^{xp}_{t})-   f_{q} (\gamma^{xp}_{0}) }{t} = \cos(\angle pxq).
  \end{align*}
Since by Lemma \ref{lem:fgamma} we know that $\gamma^{xq}$ represents the gradient of  $ -\d(q,x) r_{q}(\cdot)= \frac{\d(q,x)}{\d(p,x)}  f_{q}(\cdot)$ at $x=\gamma^{xq}_{0}$ in the sense of Definition \ref{def:gradfgamma}, we get that the rescaled geodesic $\tilde{\gamma}^{xq}$ defined by  $\tilde{\gamma}^{xq}_{t}:= \gamma^{xq}_{\frac{\d(p,x)}{\d(q,x)} t}$, $\forall t \in [0,1]$, represents the gradient of $f_{q}$ at $x=\tilde{\gamma}^{xq}_{0}$.
\\Therefore   $\angle \gamma^{xp}x\tilde{\gamma}^{xq}$ exists in the sense of Definition  \ref{def:angle-1} and $\angle \gamma^{xp}x\tilde{\gamma}^{xq}= \angle pxq$;
recalling the locality of the angle between geodesics (see  Remark \ref{rem:restrAngleGeo}), we conclude that $\angle \gamma^{xp}x \gamma^{xq}= \angle \gamma^{xp}x\tilde{\gamma}^{xq}$ exists in the sense of Definition  \ref{def:angle-1} and $\angle \gamma^{xp}x \gamma^{xq}= \angle pxq$.
 \end{proof}

\subsection{Angles in Wasserstein spaces}\label{SS:angleWasserstein}

In the  Wasserstein space, we have the notion of ``Plans representing gradients'' which is similar to the one of ``geodesic representing the gradient'' above.

\begin{definition}[Plans representing gradients, see \cite{G-O}]
Let $(X,\d,\mm)$ be a metric measure space,   $g\in S^{2}(X)$ and   $\Pi \in \mathcal{P}(C([0,1],X))$ be a test plan.  We say that $\Pi$ represents the gradient of $g$ if  
\[
\limi_{t\downarrow0}\int\frac{g(\gamma_t)-g(\gamma_0)}{t}\,\d\Pi(\gamma)\geq \frac12\int|\D g|^2(\gamma_0)\,\d\Pi(\gamma)+\frac12\lims_{t\downarrow0}\frac1t\iint_0^t|\dot\gamma_s|^2\,\d s\,\Pi(\gamma).
\]
\end{definition}

Let  $(\mu_t) \in AC^2([0,1], \mathcal{P}_2(X))$ be with uniformly bounded densities, $\Pi$ be its lifting given by Theorem \ref{prop:superposition},  and let $\varphi\in S^{2}(X)$. In case $\ms$ is infinitesimally Hilbertian,  it is proved in \cite[Theorem 4.6]{GH-C} that $\Pi$ represents the gradient of $\varphi$ if and only if 
\begin{equation}\label{eq:contgrad}
\ddt  {\Big|_{t=0}} \int_{X} f\,\d \mu_t =\int_{X} \la \nabla f, \nabla \varphi\ra\,\d\mu_0, \quad \forall f\in S^{2}(X).
\end{equation}
If \eqref{eq:contgrad} holds,  we also say that the velocity field of $\mu_t$ at time $0$ is $\nabla \varphi$. 

Combing the above technical tools with ideas from Otto's calculus \cite{O-G}, we can define the angle between two geodesics in $\ws$.
\begin{definition}[Angle between curves in $\ws$]\label{def:anglewasserstein}
Let $\ms$  be an infinitesimally Hilbertian metric measure space, let $(\mu_t), (\nu_{t}) \in AC^2([0,1], \mathcal{P}_2(X))$ be with bounded compression, and  such that $\mu_0=\nu_0=:\eta$.   Assume  there exist  lifting test  plans of $(\mu_t)$ and $(\nu_t)$ representing the gradients of $f$ and $g$ respectively,  for some $f,g\in S^{2}(X)$. Then the angle between  $\mu=(\mu_t)$ and $\nu=(\nu_t)$ at  $t=0$ is defined by
\[
[0,\pi] \ni \angle^W \mu \eta \nu:=\arccos \left( \frac{\int_{X} \la \nabla f, \nabla g \ra\,\d \eta }{\| |\D g| \|_{L^2(X,\eta)} \| |\D f| \|_{L^2(X,\eta)}} \right).
\]
The same definition makes sense if  $f, g \in S^{2}_{loc}(X)$ provided $(\mu_{t}),(\nu_{t})$ have uniformly bounded supports.
\end{definition}
From the formula \eqref{eq:contgrad}, we can see that the value of the angle does not depend on the choice of $f, g$,   but just on $(\mu_{t})$, $(\nu_{t})$.
\\

\begin{remark}[Locality of the angle in the  Wasserstein space]\label{rem:locAngleW}
The angle  $\angle^W \mu \eta \nu$ depends just on  the germs of the curves $\mu$ and $\nu$ at $t=0$; i.e., given $T_{1},T_{2}\in (0,1)$, called $\tilde{\mu}_{t}:=\mu_{T_{1}t},  \tilde{\nu}_{t}:=\nu_{T_{2}t}$  for all $t\in [0,1]$ the restrictions of $\mu, \nu$ to $[0,T_{1}]$ and $[0,T_{2}]$ respectively, it holds  $\angle \mu \eta \nu=\angle \tilde{\mu} \eta \tilde{\nu}$. Indeed let  $\Pi$, lift of the curve $(\mu_{t})_{t \in [0,1]}$, be a test plan representing the gradient of $f \in S^{2}(X)$; fix  $T\in (0,1)$ and  let $\tilde{\mu}_{t}:=\mu_{Tt}$ for every $t \in [0,1]$ be the restriction of the curve $\mu$ to $[0,T]$; called $\tilde{\Pi}$ the lift of $(\tilde{\mu}_{t})_{t \in [0,1]}$, it is easily seen that $\tilde{\Pi}$ represents the gradient of $\tilde{f}:=T f$. The claim follows.
\end{remark}

Thanks to the locality expressed in Remark \ref{rem:locAngleW}, given two curves $(\mu_{t})_{t\in [0,1]},(\nu_{t})_{t\in [0,1]}$ such that they are of bounded compression once restricted to $[0,T]$ for some $T\in (0,1)$,  we can define the angle between them as the angle between their restrictions $\tilde{\mu}, \tilde{\nu}$ to $[0,T]$. This will be always tacitly assumed throughout the paper.
\\

Let us briefly discuss the  particular but important case  when $(\mu_{t})$ and $(\nu_{t})$ are $\ws$-geodesics in a general m.m.s. $\ms$.  If $(\mu_{t})$ is a  $\ws$-geodesic with bounded compression   then any lift $\Pi$ of $(\mu_{t})$ is a test plan and moreover  is an optimal dynamical plan, i.e. $\Pi \in \opt\geo(\mu_{0},\mu_{1})$.  Moreover, as a consequence of the Metric Brenier Theorem proved in  \cite{AGS-C} (see also \cite[Theorem 5.2]{GH-C} for the present formulation), if $(\mu_{t})$ has bounded compression and $\varphi\in S^{2}(X)$ is a Kantorovich potential from $\mu_{0}$ to $\mu_{1}$,  then any lift $\Pi$ of  $(\mu_{t})$ represents the gradient of $-\varphi$. Therefore, specializing Definition \ref{def:anglewasserstein} to this case we get the following notion.

\begin{definition}[Angle between geodesics in $\ws$]\label{def:anglewassersteinGeo}
Let $\ms$  be an infinitesimally Hilbertian metric measure space, let $(\mu_t), (\nu_{t})$ be  $\ws$-geodesics with bounded compression, and  such that $\mu_0=\nu_0=:\eta$.   Assume  there exist $\varphi, \psi \in S^{2}(X)$ Kantorovich potentials from $\mu_{0}$ to $\mu_{1}$ and from $\nu_{0}$ to $\nu_{1}$ respectively. Then the angle between  $\mu=(\mu_t)$ and $\nu=(\nu_t)$ at  $t=0$ is defined by
\[
[0,\pi] \ni  \angle^W \mu \eta \nu :=\arccos \left( \frac{\int_{X} \la \nabla \varphi, \nabla \psi \ra\,\d \eta }{\| |\D \varphi| \|_{L^2(X,\eta)} \| |\D \psi| \|_{L^2(X,\eta)}} \right).
\]
The same definition makes sense if  $\varphi, \psi \in S^{2}_{loc}(X)$ provided $(\mu_{t}),(\nu_{t})$ have uniformly bounded supports.
\end{definition}

Note that, thanks to Otto calculus and \eqref{eq:contgrad},  Definition \ref{def:anglewassersteinGeo} is the analog for $\ws$ geometry of the angle between two geodesics in a general metric space in the sense of Definition  \ref{def:angle-1}.



\subsection{The case of $\rcdkn$ spaces}

In Theorem \ref{thm:def1def2} we related  the angle between three points with the angle between two geodesics, i.e.  we related  Definitions \ref{def:anglepxq} and   \ref{def:angle-1}.
Now, adding  a curvature assumption on the space, we wish to relate  Definition  \ref{def:anglewassersteinGeo} with Definition  \ref{def:anglepxq} and Definition  \ref{def:angle-1}, i.e. the angle  between two geodesics in $\ws$ with the angle between three points and the angle between two geodesics of $X$.  To this aim the next lemma will be useful.

\begin{lemma}\label{prop:well-1}
Let $\ms$ be an $\rcdkn$ space, and let $\varphi_{1}, \varphi_{2}$ be locally Lipschitz functions on $X$.   Then the functions 
\begin{align*}
F^+(x)&:=\mathop{\lim}_{\epsilon \downarrow 0}\frac{|\lip{\varphi_1+\epsilon \varphi_2}|^2(x)- |\lip{ \varphi_1}|^2(x)}{2\epsilon},  \\
F^-(x)&:=\mathop{\lim}_{\epsilon \uparrow 0}\frac{|\lip{\varphi_1+\epsilon \varphi_2}|^2(x)- |\lip{ \varphi_1}|^2(x)}{2\epsilon}
\end{align*}
are well defined  at every $x \in X$ and it holds
 $$F^+=F^-=\la \nabla \varphi_1, \nabla \varphi_2 \ra, \quad \mm\text{-a.e. }.$$ 
\end{lemma}

\begin{proof}
From the definition of local Lipschitz constant we know that the function $\epsilon \mapsto |\lip{\varphi_1+\epsilon \varphi_2}|^2(x)$ is convex for any $x$. 
Consider the function 
\[
F_\epsilon(x):=\epsilon \mapsto \frac{|\lip{\varphi_1+\epsilon \varphi_2}|^2(x)- |\lip{ \varphi_1}|^2(x)}{2\epsilon},
\]
and observe that $\epsilon\mapsto F_\epsilon(x)$ is non-decreasing on $(-\infty,0)$ and $(0, +\infty)$ for any fixed $x$. Hence $F^+$ and $F^-$ are well-defined for any point $x\in X$ as
\begin{align*}
F^+(x):= \mathop{\inf}_{\epsilon >0 } F_{\epsilon}(x)=\mathop{\lim}_{\epsilon \downarrow 0}  F_{\epsilon}(x), \quad  F^-(x):=\mathop{\sup}_{\epsilon <0} F_{\epsilon}(x) = \mathop{\lim}_{\epsilon \uparrow 0} F_{\epsilon}(x).
\end{align*}
Since  $\ms$ is a $\rcdkn$ metric measure space, it holds a local Poincar\'e inequality and it is locally doubling. Then, from  \cite[Theorem 6.1]{C-D},  we know $\lip{f}(x)=|\D f|(x)$ for $\mm$-a.e. $x \in X$. The definition of infinitesimal Hilbertian space and  of $\la \nabla \varphi_1, \nabla \varphi_2 \ra$, then gives 
\[
\la \nabla \varphi_1, \nabla \varphi_2 \ra=\mathop{\mathrm{ess~inf}}_{\epsilon>0} \frac{|\D(\varphi_1+\epsilon \varphi_2)|^2- |\D \varphi_1|^2}{2\epsilon}=\mathop{\mathrm{ess~inf}}_{\epsilon>0} F_\epsilon
\]
and
\[
\la \nabla \varphi_1, \nabla \varphi_2 \ra=\mathop{\mathrm{ess~sup}}_{\epsilon<0} \frac{|\D(\varphi_1+\epsilon \varphi_2)|^2- |\D \varphi_1|^2}{2\epsilon}=\mathop{\mathrm{ess~sup}}_{\epsilon<0} F_\epsilon.
\]
Hence 
\[
\la \nabla \varphi_1, \nabla \varphi_2 \ra=\mathop{\mathrm{ess~inf}}_{\epsilon>0} F_\epsilon=\mathop{\mathrm{ess~sup}}_{\epsilon<0} F_\epsilon.
\]
In particular, we infer $F^+=F^-=\la \nabla \varphi_1, \nabla \varphi_2 \ra$ $\mm$-a.e..
\end{proof}

In the next result we relate  Definition  \ref{def:anglewassersteinGeo} with Definition  \ref{def:anglepxq}.  Before stating it, let us recall \cite[Theorem 1.1]{GRS-O} that  if $\ms$ is an $\rcdkn$ m.m.s., $\mu_{0},\mu_{1}\in \mathcal{P}_{2}(X)$ with $\mu_{0}\ll \mm$, then there exists a unique $\ws$ geodesic connecting $\mu_{0}$ and $\mu_{1}$; let us mention that the same result holds more generally for essentially non-branching m.m.s satisfying the weaker ${\textrm {MCP}}(K,N)$ condition \cite{CM-O}. 

\begin{proposition}\label{prop:AngleWpoints}
Let $\ms$ be an $\rcdkn$ m.m.s.  and fix $p,q \in X$.  For every $x \in X$ and  $R>0$  let $\mu^{R}_{0}=\nu^{R}_{0}=\eta^{R}:=\frac{1}{\mm(B_{R}(x))} \mm \llcorner B_{R}(x)$ and let $\mu^{R}:=(\mu^{R}_{t})_{t\in [0,1]},\nu^{R}:=(\nu^{R}_{t})_{t\in [0,1]}$ be the unique $\ws$-geodesics from 
$\mu^{R}_{0}$ to $\delta_{p}$ and from $\nu^{R}_{0}$ to $\delta_{q}$ respectively. Then
\begin{equation}
\angle pxq = \lim_{R\downarrow 0} \angle^{W}  \mu^{R} \eta^{R} \nu^{R}, \quad \text{for  $\mm$-a.e. $x \in X$}.
\end{equation}
\end{proposition}

\begin{proof}
Calling as usual $r_{p}(\cdot):=\d(p,\cdot),r_{q}(\cdot):=\d(q,\cdot)$, Lemma \ref{prop:well-1} implies
\begin{align*}
\cos (\angle pxq) &:=\lim_{\varepsilon \to 0} \frac{|\llip(r_{p}+\varepsilon r_{q})|^{2}(x) - |\llip(r_{p})|^{2}(x)  }{2 \varepsilon}= \la \nabla r_{p}, \nabla r_{q} \ra (x), \; \text{$\mm$-a.e. $x \in X$}.
\end{align*}
On the other hand, it is easily seen that $\varphi(\cdot):= \frac{1}{2} r_{p}(\cdot)^{2}, \psi(\cdot):= \frac{1}{2}  r_{q}(\cdot)^{2}$ are Kantorovich potentials from $\eta=\mu^{R}_{0}$ to $\delta_{p}$ and from $\eta=\nu^{R}_{0}$ to $\delta_{q}$ respectively. Moreover the geodesics $(\mu^{R}_{t}), (\nu^{R}_{t})$ have uniformly bounded supports, and bounded compression once restricted to $[0,1-\delta]$ for every $\delta\in (0,1)$, see for instance \cite[Corollary 1.7]{GRS-O}. Then,  Definition \ref{def:anglewassersteinGeo} yields
\begin{align*}
\lim_{R\downarrow 0}\cos( \angle^{W} \mu^{R} \eta^{R} \nu^{R})&:=  \lim_{R\downarrow 0} \frac{\int_{X} \la \nabla \varphi, \nabla \psi \ra\,\d \eta }{\| |\D \varphi| \|_{L^2(X,\eta)} \| |\D \psi| \|_{L^2(X,\eta^{R})}}\\
&=  \lim_{R\downarrow 0} \frac{\frac{1}{\mm(B_{R}(x))}\int_{B_{R}(x)} r_{p}(y) \, r_{q}(y) \, \la  \nabla r_{p}, \nabla r_{q} \ra (y) \,\d \mm(y) }{\| r_{p} \|_{L^2(X,\eta^{R})} \| r_{q} \|_{L^2(X,\eta^{R})}}\\
&= \la \nabla r_{p}, \nabla r_{q} \ra (x)  \quad \text{for  $\mm$-a.e. $x \in X$}. 
\end{align*}
The combination of the two formulas gives the claim.
\end{proof}

\begin{remark}
For  uniformity with the rest of the paper we decided to state Proposition \ref{prop:AngleWpoints} for $\rcdkn$ spaces, but using the results of \cite{CM-O}  the same conclusion holds for essentially non-branching Lipschitz-infinitesimally Hilbertian spaces satisfying $\mathrm{MCP}(K,N)$. 
\end{remark}

In the  next result  we relate Definition  \ref{def:angle-1} with the optimal transport picture.

\begin{proposition}\label{prop:well-2}
Assume that $\ms$ is an $\rcdkn$ metric measure space. Let  $(\mu^1_t)$ and $(\mu^2_t)$ be  $\mathcal{W}_2$-geodesics with bounded compression and with $\mu^{1}_{0}=\mu^{2}_{0}=:\eta$;  let $\Pi_{1},\Pi_{2}\in \opt\geo(X)$ be  corresponding lifts.
Then, for $i=1,2$, we can find $\Gamma_{i}\subset \geo(X)$ with $\Pi_{i}(\Gamma_{i})=1$, such that  for  $\eta$-a.e. $x$ there exist unique geodesics $\gamma^{x,i}\in \Gamma_{i}$ with $\gamma^{x,i}_{0}=x$,  and the angle $\angle \gamma^{x,1} x \gamma^{x,2}$  exists according to the Definition \ref{def:angle-1}. Moreover
\begin{align}
\cos \angle \gamma^{x,1} x \gamma^{x,2}&= \cos \angle \gamma^{x,2} x \gamma^{x,1}=\lim_{t \downarrow 0}  \frac {\varphi_1(\gamma^{x,2}_{t})-\varphi_1(\gamma^{x,2}_{0})}{t \;  \llip(\varphi_{1})(x)  \; |\dot{\gamma}^{x,2}|}=\lim_{t \downarrow 0} \frac {\varphi_2(\gamma^{x,1}_{t})-\varphi_2(\gamma^{x,1}_{0}) }{t \;  \llip(\varphi_{2}) (x)\; |\dot{\gamma}^{x,1}|} \nonumber\\
&=\frac{\la \nabla \varphi_1, \nabla \varphi_2\ra(x)}{\llip(\varphi_{1})(x)\; \llip(\varphi_{2})(x) }, \quad \text{for $\eta$-a.e. $x$}, \label{eq:well-0}
\end{align}
where $-\varphi_i \in S^{2}(X)$ is any locally Lipschitz Kantorovich  potential from $\eta=\mu^i_0$ to $\mu^i_1$. 
\end{proposition}
\begin{proof}
From \cite{AGS-M,GH-C} we know that any lift  $\Pi_i$ of $(\mu_t^{i})$ represents the gradient of  $ \varphi_i$, for $i=1,2$,  i.e:
\begin{equation*}
\lim_{t\downarrow0}\int\frac{\varphi_i(\gamma_t)-\varphi_i(\gamma_0)}{t}\,\d\Pi_i(\gamma)\geq \frac12\int|\D \varphi_{i}|^2(\gamma_0)\,\d\Pi_i(\gamma)+\frac12\lims_{t\downarrow0}\frac1t\iint_0^t|\dot\gamma_s|^2\,\d s\,\d\Pi_i(\gamma).
\end{equation*}
From  \cite[Proposition 3.11]{G-O} we then get  for $i=1,2$:
\begin{equation}\label{eq:well-2}
\lim_{t\downarrow 0} \frac{\varphi_i(\gamma_{t})-\varphi_i(\gamma_{0})}{t}=\frac{1}{2}|\D\varphi_i|^2(\gamma_{0}) +\frac12 |\dot{\gamma}|^2=\frac{1}{2}|\lip{\varphi_i}|^2(\gamma_{0}) +\frac12 |\dot{\gamma}|^2,  \; \Pi_{i} \text{-a.e.  }\gamma.
\end{equation}
In other words, for $\Pi_{i}$-a.e. $\gamma$, we have that $\gamma$ represents $\nabla \varphi_{i}$ at $\gamma_{0}$, $i=1,2$.
\\For any $\epsilon >0$, consider the function $\varphi_1+\epsilon \varphi_2$ and observe that 
\begin{equation}\label{eq:well-3}
\limstz  \frac{(\varphi_1+\epsilon \varphi_2)(\gamma_{t})-(\varphi_1+\epsilon \varphi_2)(\gamma_{0})}{t} \leq \frac{1}{2}|\lip{\varphi_1+\epsilon \varphi_2}|^2(\gamma_{0}) +\frac12 |\dot{\gamma}|^2, \quad  \forall \gamma\in \geo(X).
\end{equation}
The difference between \eqref{eq:well-3} and \eqref{eq:well-2}, for $i=1$, gives
\begin{equation}\label{eq:well-4}
\epsilon  \limstz  \frac{ \varphi_2(\gamma_{t})-\varphi_2(\gamma_{0})}{t} \leq \frac{|\lip{\varphi_1+\epsilon \varphi_2}|^2(\gamma_{0})-|\lip{\varphi_1}|^2(\gamma_{0})}{2},\quad \Pi_{1} \text{-a.e.  }\gamma .
\end{equation}
Multiplying by $\epsilon^{-1}>0$  both sides of \eqref{eq:well-4} yields
\[
 \limstz  \frac{ \varphi_2(\gamma_{t})-\varphi_2(\gamma_{0})}{t} \leq \frac{|\lip{\varphi_1+\epsilon \varphi_2}|^2(\gamma_{0})-|\lip{\varphi_1}|^2(\gamma_{0})}{2\epsilon},\quad \Pi_{1} \text{-a.e.  }\gamma .
\]
Letting $\epsilon \downarrow 0$ and using Lemma \ref{prop:well-1},  we infer
\[
\limstz  \frac{ \varphi_2(\gamma_{t})-\varphi_2(\gamma_{0})}{t} \leq \la \nabla \varphi_2, \nabla \varphi_1\ra(\gamma_{0}), \quad \Pi_{1} \text{-a.e.  }\gamma .
\]
Following verbatim the same arguments after \eqref{eq:well-2}, but now for  $\epsilon <0$,  gives
\[
\limitz  \frac{ \varphi_2(\gamma_{t})-\varphi_2(\gamma_{0})}{t} \geq \la \nabla \varphi_2, \nabla \varphi_1\ra(\gamma_{0}), \quad \Pi_{1} \text{-a.e.  }\gamma .
\]
 Since from \cite[Theorem 3.4]{GRS-O} we can find $\Gamma_{i}\subset \geo(X)$ with $\Pi_{i}(\Gamma_{i})=1$, such that  for  $\eta$-a.e. $x$ there exists unique geodesics $\gamma^{x,i}\in \Gamma_{i}$ with $\gamma^{x,i}_{0}=x$, it follows that 
\begin{equation}\label{eq:pfOTCurveFinal}
\lim_{t\downarrow 0}  \frac{ \varphi_2(\gamma^{x,1}_{t})-\varphi_2(\gamma^{x,1}_{0})}{t}=  \la \nabla \varphi_2, \nabla \varphi_1\ra(x), \quad \eta\text{-a.e. }x. 
\end{equation}
Recalling from  \eqref{eq:well-2}   that $\Pi_{2}$-a.e. $\gamma$ represents the gradient of $ \varphi_{2}$ at $\gamma_{0}$, we get that for $\eta$-a.e. $x$ the geodesic $\gamma^{x,2}$ represents the gradient of $\varphi_{2}$ at $x$. Therefore \eqref{eq:pfOTCurveFinal} proves that for $\eta$-a.e. $x$ the angle $\angle \gamma^{x,1} x \gamma^{x,2}$ exists according to Definition \ref{def:angle-1} and coincides with  $\arccos\big(\la \nabla \varphi_2, \nabla \varphi_1\ra(x)\big)$. With the same arguments, just exchanging $i=1$ with $i=2$, we get that also  $\angle \gamma^{x,2} x \gamma^{x,1}$ exists for $\eta$-a.e. $x$, and that the identities \eqref{eq:well-0} hold.
\end{proof}

\section{The cosine formula for angles in $\rcdkn$ spaces}
The goal of this section is to prove Theorem \ref{thm:cosine}, stating that the cosine formula holds for the angle between two geodesics in an $\rcdkn$ space.
The first lemma states  the almost everywhere uniqueness and extendability  of geodesics in $\rcdkn$ spaces; this fact is already present in the literature under slightly different formulations so we just briefly sketch the proof. 

\begin{lemma}\label{lem:ExUnGeo}
Let $\ms$ be an $\rcdkn$ space for some $K \in \R, N\in (1,\infty)$, and fix $p,q \in X$. Then for $\mm$-a.e. $x$ there exist unique geodesics $\gamma^{xp},\gamma^{xq} \in \geo(X)$ such that
\begin{itemize}
\item  $\gamma^{xp}_{0}=\gamma^{xq}_{0}=x$, $\gamma^{xp}_{1}=p$, $\gamma^{xq}_{1}=q$,
\item both $\gamma^{xp}$ and $\gamma^{xq}$ are extendable to  geodesics  $\tilde{\gamma}^{xp}$ and $\tilde{\gamma}^{xq}$  having $x$ as interior point; in other words there exist $\tilde{\gamma}^{xp},\tilde{\gamma}^{xq} \in \geo(X)$ and $\bar{t}\in (0,1)$ such that  $\tilde{\gamma}^{xp}_{\bar{t}},\tilde{\gamma}^{xq}_{\bar{t}}=x$ and $\tilde{\gamma}^{xp}_{[\bar{t},1]}= \gamma^{xp}_{[0,1]},\tilde{\gamma}^{xq}_{[\bar{t},1]}=\gamma^{xq}_{[0,1]}.$ 
 \end{itemize}
\end{lemma}

\begin{proof}
\textbf{Step 1}.  $ \forall p \in X$, $\mm$-a.e. $x\in X$ is an interior point of a geodesic with end point at $p$.
\\Fix $p \in X$ and $R>0$. Consider
 $$\mu_{0}:= \frac{1}{\mm(B_{R}(p))}  \; \mm \llcorner B_{R}(p) \quad \text{ and } \quad \mu_{1}:=\delta_{p}. $$ 
 Analyzing the optimal transport from $\mu_{0}$ to $\mu_{1}$ by following verbatim the proof of  \cite[Lemma 3.1]{GMR-E} (i.e. use Jensen's inequality and the convexity property of the entropy granted by the curvature condition), we get that for $\mm$-a.e. $x \in B_{R}(0)$ there exists a geodesic $\gamma \in \geo(X)$ such that $\gamma_{1}=p$ and $\gamma_{t}=x$, for some $t \in (0,1)$. The claim then follows by the arbitrariness of $R>0$. 
 \\
 
 \textbf{Step 2}. $ \forall p \in X$, $\mm$-a.e. $x\in X$ there exists a unique geodesic from $x$ to $p$.
 \\The uniqueness of  geodesics connecting a fixed $p \in X$ and $\mm$-a.e. $x \in X$ is a consequence of \cite[Theorem 3.5]{GRS-O} applied to the optimal transportation from the measures $\mu_{0}$, $\mu_{1}$ above.
 
 \textbf{Step 3}. Applying steps 1 and 2 to $p$ and $q$, since the union of two negligible sets is still negligible, the thesis follows.
 \end{proof}

The next lemma will be useful to get good estimates on harmonic approximations of distance functions.

\begin{lemma}\label{lemma:goodfunction}
Let  $B$ be a unit ball in an $\rcdkn$ metric measure space $\ms$, $K \in \R, N\in (1,\infty)$. Then  there exists a  function $G:B\to \R$ with  $G \in {\rm D}({\bf \Delta},B)$ such that 
\[
{\bf \Delta}G \llcorner B=(\Delta G \llcorner B) \, \mm \llcorner B, \quad \Delta G \llcorner B = 1, \quad 0\leq G\leq C \quad \text{on B},
\]
where $C=C(K,N)>0$ is a constant which  depends only on $K$ and  $N$.  
\end{lemma}
\begin{proof}
Since  $\ms$ is a $\rcdkn$ metric measure space, it satisfies  a local (1-2)-Poincar\'e inequality and it is locally doubling.  It is also  known \cite[Remark 6.9 and Theorem 6.10]{AGS-M} that the metric $\d$ is induced by the Dirichlet form $f\mapsto \int |\D f|^2\,\d \mm$. Therefore the standing assumptions of \cite{BM-S} are fulfilled and    from \cite[Corollary 1.2] {BM-S} we know that for any $f\in L^p(B,\mm)$, $p> 2$ , there exists a function $u_f \in W^{1,2}_0(B)$ such that
\[
\int_{B} \la \nabla u_f,  \nabla v \ra\,\d \mm=\int_B fv\,\d \mm
\]
for any $v \in W^{1,2}_0(B)$. In other words, we know $u_f \in {\rm D}({\bf \Delta}, B)$ and 
\[
\Delta u_f= f \quad \mm\text{-a.e. }.
\]
Furthermore, from \cite[Theorem 4.1]{BM-S} we know 
\[
\mathop{\sup}_{B} |u_f| \leq c \,\mm(B)^{-\frac1p} \| f\|_{L^p(B, \mm)}
\]
where $c$ only depends on  the constants  in  the Poincar\'e inequality and in the  doubling condition.  In our case, $c$ only depends on $N$ and $K$.
\\Now, choosing $f= 1$ on $B$, we  get that  $G:=u_f+ c$ satisfies the thesis with $C=2c$.
\end{proof}

Using Lemma \ref{lemma:goodfunction}, in the next proposition we prove a key estimate in order to establish the cosine formula for angles.

\begin{proposition}\label{lemma-harmonicapprox}
Let $\ms$ be an $\rcdkn$  metric measure space, for some $K \in \R, N\in (1,\infty)$, and fix $x_{0}\in X$.  Let $R\geq 2$,   $p, \hat{p} \in X$ such that $\d(x_{0},p)+\d(x_0, \hat{p})=\d(p, \hat{p})$,  and $\d(x_{0},p), \d(x_0, \hat{p}) \geq R$.  We denote $b_{p}(\cdot):=\d(p,\cdot)-\d(p,x_{0})$ and $b_{\hat{p}}(\cdot):=\d(\hat{p},\cdot)-\d(\hat{p},x_{0})$. Assume that there exists a function $\Phi(R|K,N)$ satisfying  $\lim_{R\to +\infty}\Phi(R|K,N)=0$ for fixed $K,N$, such that $0\leq { b}_{\hat{p}}(x)+{ b}_{{p}}(x)\leq \Phi(R|K,N)$ for any $x\in  B_{1}(x_{0})$.

 Then there exists a harmonic approximation ${\bf b}_{p}$ of $b_{p}$ with  the following properties:
\begin{enumerate}
\item ${\bf b}_{p}-b_{p} \in W^{1,2}_{0}( B_{1}(x_{0}))$,  ${\bf b}_{p} \in  {\rm D}({\bf \Delta}, B_{1}(x_{0}))$ with ${\Delta} {\bf b}_{p} \llcorner  B_{1}(x_{0})=0$,
\item  it holds 
\begin{equation}\label{eq:HarmApprox}
\|{\bf b}_{p}- b_{p}\|_{L^\infty(B_1(x_{0}))}+ \frac{1}{\mm(B_{1}(x_{0}))} \int_{B_{1}(x_{0})}|\D( {\bf b}_{p}-  b_{p})| ^{2} \, \d \mm \leq \Psi(R| K, N),
\end{equation}
 where $\Psi:\R^{3}\to \R_{>0}$ satisfies  $\lim_{R\to +\infty}\Psi(R|K,N)=0$ for fixed $K,N$.
\end{enumerate}
\end{proposition}

\begin{proof}
From  Proposition \ref{prop-harmonic} we know there exists  $ {\bf b}_{p}$ satisfying \emph{(1)} of the thesis. Similarly, we can find a harmonic approximation $ {\bf b}_{\hat{p}}$ 
of $b_{\hat{p}}$.
\\ We are then left to show the validity of the estimate \eqref{eq:HarmApprox}. To this aim, let $G:B_{1}(x_{0})\to \R_{\geq 0}$  be given by Lemma  \ref{lemma:goodfunction}, so that  
\[
{\bf \Delta}G \llcorner B_{1}(x_{0})=(\Delta G \llcorner B_{1}(x_{0})) \, \mm  \llcorner B_{1}(x_{0}) , \quad \Delta G \llcorner B_{1}(x_{0}) = 1, \quad 0\leq G\leq C \quad \text{on }B_{1}(x_{0}),
\]
where $C(K,N)$ depends only on $K,N$ and in particular is  independent of $R$.
\\From Laplacian Comparison Theorem \ref{thm:LapComp}   we know  that  ${b}_{p}, {b}_{\hat{p}}\in {\rm D}({\bf \Delta}, B_{1}(x_{0}))$ and 
\begin{equation}\label{eq:compDeltabp}
{\bf \Delta} {b}_{p} \llcorner B_{1}(x_{0}) \leq \Psi(R|K,N) \, \mm
\end{equation}
and
\begin{equation}\label{eq:compDeltabp2}
{\bf \Delta} {b}_{\hat{p}} \llcorner B_{1}(x_{0}) \leq \Psi(R|K,N) \, \mm
\end{equation}
for some suitable  $\Psi:\R^{3}\to \R_{>0}$ satisfying $\lim_{R\to +\infty}\Psi(R|K,N)=0$ for fixed $K,N$. Then we have 
\begin{align*}
\Delta ({ b}_{p}-{ \bf b}_{p}-\Psi G) \llcorner B_{1}(x_{0})\leq 0, \quad  \Delta (-{ b}_{\hat{p}}+{ \bf b}_{\hat{p}}+\Psi G)  \llcorner B_{1}(x_{0})\geq 0.
\end{align*}
Applying the comparison statement of  Proposition \ref{prop-harmonic} to  $(-{ b}_{p}+{ \bf b}_{p})+\Psi G$ we get that 
\begin{equation}\label{eq:bpinf1}
-{ b}_{p}+{ \bf b}_{p} \leq 2 \Psi \sup_{B_{1}(x_{0})} |G| \leq 2 C(K,N) \Psi, \quad \mm\text{-a.e. on }B_{1}(x_{0}). 
\end{equation}
Analogously, applying  the comparison statement of  Proposition \ref{prop-harmonic} to  $-{ b}_{\hat{p}}+{ \bf b}_{\hat{p}}+\Psi G$ we get 
\begin{equation*}
-{ b}_{\hat{p}}+{ \bf b}_{\hat{p}} \leq  2 \Psi \sup_{B_{1}(x_{0})} |G| \leq  2 C(K,N) \Psi,  \quad \mm\text{-a.e. on }B_{1}(x_{0}). 
\end{equation*}
By assumption,  we know  there exists a function $\Phi(R|K,N)$ satisfying  $\lim_{R\to +\infty}\Phi(R|K,N)=0$ for fixed $K,N$, such that $0\leq { b}_{\hat{p}}(x)+{ b}_{{p}}(x)\leq \Phi(R|K,N)$ for any $x\in  B_{1}(x_{0})$. Using maximum principle of  Proposition \ref{prop-harmonic}, we know
\begin{equation*}
0\leq {\bf b}_{\hat{p}}(x)+{\bf b}_{{p}}(x)\leq \Phi(R|K,N)
\end{equation*}
for any $x\in  B_{1}(x_{0})$. The combination of the last three estimates gives
\begin{align}\label{eq:otherside}
 b_{p}- {\bf b}_{p}  &= ( b_{p} +b_{\hat{p}})- ({\bf b}_{p} + {\bf b}_{\hat{p}})+ (-{ b}_{\hat{p}}+{ \bf b}_{\hat{p}})   \leq \Phi(R|K,N) + 2 C(K,N) \Psi(R|K,N), \quad \text{on } B_{1}(x_{0}).
\end{align}
Putting together \eqref{eq:bpinf1} and  \eqref{eq:otherside}, we get
\begin{equation}\label{eq:bpinf}
\|{ b}_{p}-{ \bf b}_{p}\|_{L^{\infty}(B_{1}(x_{0}))} \leq  2 \, C(K,N) \, \Psi(R|K,N)+ \Phi(R|K,N). 
\end{equation}
Next, write $B=B_{1}(x_{0})$ for short. Recalling that ${\bf  \Delta b}_{p} \llcorner B =0$,  combining \eqref{eq:compDeltabp} with \eqref{eq:bpinf} and using that $({\bf b}_{p} -b_{p}) \in W^{1,2}_{0}(B)$ in order to integrate by parts,   we obtain
\begin{align}
\int_{B} |\D ({\bf b}_{p} -b_{p})|^2 \,	\d \mm& =- \int_{B}  ({\bf b}_{p} -b_{p})  \, \d \left({\bf \Delta}({\bf b}_{p} -b_{p}) \right) \nonumber  \\
&=  \int_{B}  \big( \|{ b}_{p}-{ \bf b}_{p}\|_{L^{\infty}(B)} + ({\bf b}_{p} -b_{p}) \big)  \, \d \left( {\bf \Delta}  b_{p} \right) + \|{ b}_{p}-{ \bf b}_{p}\|_{L^{\infty}(B)} \int_{B}   \d  \left({\bf  \Delta}({\bf b}_{p} -b_{p}) \right)  \nonumber \\
&\leq   \int_{B}  \big( \|{ b}_{p}-{ \bf b}_{p}\|_{L^{\infty}(B)} + ({\bf b}_{p} -b_{p}) \big)  \, \Psi(R|K,N) \, \d \mm  \nonumber \\
&\leq 2   \|{ b}_{p}-{ \bf b}_{p}\|_{L^{\infty}(B)}  \; \mm(B) \;  \Psi(R|K,N) \nonumber \\
& \leq  2 \big(  2 C(K,N) \, \Psi(R|K,N)+\,\Phi(R|K,N)  \big) \,\mm(B) \, \Psi(R|K,N), \label{eq:gbbL2}
\end{align}
where we used that, since  $({\bf b}_{p} -b_{p}) \in W^{1,2}_{0}(B)$, it holds
$$ \int_{B} \d  \left({\bf \Delta}({\bf b}_{p} -b_{p}) \right)= \int_{B} 1 \; \d  \left({\bf \Delta}({\bf b}_{p} -b_{p}) \right)= -\int_{B}  \la \nabla ({\bf b}_{p} -b_{p}), \nabla 1 \ra \, \d \mm =0. $$
Summing up \eqref{eq:bpinf} and  \eqref{eq:gbbL2} we get  \eqref{eq:HarmApprox} by renaming with  $\Psi(R|K,N)$ the quantity $ 2 \big(  2 C(K,N) \, \Psi(R|K,N)+\,\Phi(R|K,N)  \big)  \, \Psi(R|K,N)+ 2 \, C(K,N) \, \Psi(R|K,N)+ \Phi(R|K,N)$.
\end{proof}

\begin{theorem}\label{thm:cosine}
Let $\ms$ be an $\rcdkn$ space for some $K \in \R, N\in (1,\infty)$, and fix $p,q \in X$. Then for  $\mm$-a.e. $x \in X$  let  $\gamma^{xp},\gamma^{xq} \in \geo(X)$ be the unique geodesics from $x$ to $p$ and from $x$ to $q$ given by Lemma  \ref{lem:ExUnGeo}.
We may also assume that the tangent cone at $x$ is unique and isomorphic as m.m. space to $(\R^{k},\d_{E},  \mathcal{L}_k)$, for some $k=k(x) \in \N \cap [1,N]$.
Let $r_{i}\downarrow 0$ be any sequence, $\bar{p}, \bar{q}\in \R^{k}$ be the limit points of $\gamma^{xp}(r_{i}), \gamma^{xq}(r_{i})$ under the rescalings $(X, r_{i}^{-1} \d, \mm^{x}_{r_{i}}, x)$ which converge to  $(\R^{k},\d_{E},  \mathcal{L}_k, O)$  in p-mGH sense.
Then
\begin{equation}\label{eq::anglepxpgamma}
  \angle \gamma^{xp} x \gamma^{xq}=  \angle p x q=   \angle \bar{p}O\bar{q}= \lim_{t \downarrow 0} \arccos \frac{2t^2-\d^2(\gamma^{xp}_{t}, \gamma^{xq}_{t})}{2t^{2}} ,\quad  \text{for  $\mm$-a.e. $x$}.
\end{equation}
\end{theorem}

\begin{proof}
\textbf{Step 1}.
Fix $p,q \in X$. Combining Theorem \ref{RCD-reg}, Remark \ref{rem:LIH}, Theorem \ref{thm:def1def2}  and Lemma \ref{lem:ExUnGeo} we get that  for $\mm$-a.e. $x \in X$
\begin{itemize}
\item  we can find unique geodesics     $\gamma^{xp},\gamma^{xq} \in \geo(X)$ such that  $\gamma^{xp}_{0}=\gamma^{xq}_{0}=x$, $\gamma^{xp}_{1}=p$, $\gamma^{xq}_{1}=q$, and both $\gamma^{xp},\gamma^{xq}$ are extendable beyond $x$ in the sense of  Lemma \ref{lem:ExUnGeo},  so we  can assume that $\gamma^{xp},\gamma^{xq}$ could be extended to $\hat{p}:=\gamma^{xp}_{-\varepsilon}, \hat{q}:=\gamma^{xq}_{-\varepsilon}$ respectively, for some $\varepsilon>0$,
\item both the angles $\angle pxq$ and $\angle \gamma^{xp} x \gamma^{xq}$ exist in the sense of Definitions \ref{def:anglepxq}, \ref{def:angle-1}  respectively,  and  $\angle pxq=\angle \gamma^{xp} x \gamma^{xq}$,
\item $x\in X$ is a Lebesgue point for $\la \nabla r_{p}, \nabla r_{q}\ra$ so that  
\begin{equation}\label{eq:LebsesqueAngle}
 \cos \angle pxq :=  \la \nabla  r_{p}, \nabla r_{q}\ra (x)= \lim_{r\downarrow 0} \frac 1{\mm(B_{r}(x))}\int_{B_{r}(x)} \la \nabla  r_{p}, \nabla r_{q}\ra\,\d \mm,
 \end{equation}
\item the tangent to $X$ at $x$ is unique and euclidean.
\end{itemize}
From the locality of the angle (see Remark \ref{rem:restrAngleGeo}) we know that 
\begin{equation}\label{eq:localProof}
 \angle pxq=\angle \gamma^{xp} x \gamma^{xq} = \angle  (\gamma^{xp}|_{0}^{s}) \,  x  \,    (\gamma^{xq}|_{0}^{t}), \quad \forall s,t \in (0,1),
\end{equation}
where $\left(\gamma |_{0}^{s}\right)_{t}:= \gamma_{st}$ for all $t \in [0,1]$.
\\Let $r_{i}\downarrow 0$ be any sequence and let  $(X, r_{i}^{-1} \d, \mm^{x}_{r_{i}}, x)$ be the corresponding sequence of rescaled spaces. Since by assumption $x$ is regular, we know that  $(X, r_{i}^{-1} \d, \mm^{x}_{r_{i}}, x)$ p-mGH converge
to $(\R^{k}, \d_{E}, {\mathcal L}_{k}, O)$ for some $k=k(x)\in \N \cap [1,N]$. Since by assumption both $\gamma^{xp}$ and $\gamma^{xq}$ are extendable beyond $x$, they converge in p-GH sense to half lines  $\ell_{p}, \ell_{q}$ in $\R^{k}$ such that
$O\in \ell_{p}\cap \ell_{q}$ and both $\ell_{p},\ell_{q}$ are extendable to full lines of $\R^{k}$. We parametrize such half lines on $[0,+\infty)$   such that for every $t>0$ one has that  $\ell_{p}(t), \ell_{q}(t)$ are the limit points of $\gamma^{xp}(r_{i} t)$, $\gamma^{xq}(r_{i} t)$ respectively.     Denote by $\bar{p}=\ell_{p}(1), \bar{q}=\ell_{q}(1)\in \R^{k}$ be the limit points of $\gamma^{xp}(r_{i}), \gamma^{xq}(r_{i})$. By the uniqueness of the tangent space, the parametrized half lines $\ell_{p}, \ell_{q}$ and the points $\bar{p},\bar{q} \in \R^{k}$ do not depend on the choice of the   rescaling sequence $(r_{i})$.
\\
Let $\hat{\ell}_{p}, \hat{\ell}_{q}$ be the half lines in $\R^{k}$ antipodal to   ${\ell}_{p}, {\ell}_{q}$ respectively; in other words $\ell_{p}\cup \hat{\ell}_{p}$ and $\ell_{q}\cup \hat{\ell}_{q}$ are straight lines in $\R^{k}$ intersecting at $O$.  We parametrize $\hat{\ell}_{p}, \hat{\ell}_{q}$ on $(-\infty,0]$   such that $\frac{d}{dt}|_{t=t_{1}} \hat{\ell}_{p}(t)=\frac{d}{dt}|_{t=t_{2}}\ell_{p}(t), \frac{d}{dt}|_{t=t_{1}} \hat{\ell}_{q}(t)=\frac{d}{dt}|_{t=t_{2}}\ell_{q}(t)$, for all $t_{1}<0<t_{2}$.
\\

\textbf{Step 2}. We claim that 
\begin{equation}\label{eq:pfstep2}
\angle \gamma^{xp} x \gamma^{xq}= \angle p x q = \angle \bar{p} O \bar{q} = \angle \ell_{p} O \ell_{q}.
\end{equation}
Since the first identity is true by construction, and the last is trivially true because the ambient space is $\R^{k}$, it is enough to show that   $\angle p x q = \angle \bar{p} O \bar{q}$.
Given any sequence of rescalings $r_{i}\downarrow 0$,  let $\d^{i}(\cdot, \cdot):= \frac{1}{r_{i}} \d(\cdot, \cdot)$ and define
\begin{equation}\label{eq:defbip}
b^{i}_{p}(\cdot):= \d^{i}(p, \cdot)-\d^{i}(p,x), \qquad  b^{i}_{q}(\cdot):= \d^{i}(q, \cdot)-\d^{i}(q,x).
\end{equation}
Set also $b^{\infty}_{p},b^{\infty}_{q}:\R^{k} \to \R$ to be the Busemann functions associated to $\ell_{p},\ell_{q}$, i.e.  
\begin{equation}\label{eq:defbipinf}
b^{\infty}_{p}(\cdot):= \lim_{t\to +\infty}  t- \d_{E}(\ell_{p}(t), \cdot) , \qquad b^{\infty}_{q}(\cdot):= \lim_{t\to +\infty}  t- \d_{E}(\ell_{q}(t), \cdot).
\end{equation}
Since by construction we know that $(X,\d^{i},x)\to (\R^{k}, \d_{E}, O)$ in p-GH sense, $\gamma^{xp}_{r_{i} t} \to \ell_{p}(t)$ and  $\gamma^{xq}_{r_{i} t} \to \ell_{q}(t)$ for every $t>0$, it follows that
\begin{equation}\label{eq:bibiinfpoint}
b^{i}_{p}\to  b^{\infty}_{p}, \quad  b^{i}_{q} \to  b^{\infty}_{q} \quad \text{pointwise in the sense of Definition \ref{def:pointConv}.}
\end{equation} 
More strongly, since $b^{i}_{p}$ are all Lipschitz with unit Lipschitz constant, by an Arzel\'a-Ascoli procedure (see for instance   in \cite[Proposition 2.12]{MN-S}) we get that the convergences are uniform on bounded subsets, in the sense of  Definition \ref{def:pointConv}.  In particular, since the measures $\mm^{x}_{r_{i}}$ are converging weakly to ${\mathcal L}_{k}$ we get that
\begin{equation}\label{eq:bibinfL2}
b^{i}_{p}\to  b^{\infty}_{p}, \quad  b^{i}_{q} \to  b^{\infty}_{q} \quad \text{strongly in $L^{2}$ in the sense of Definition \ref{def:L2conv}.}
\end{equation} 
Define $b^{i}_{\hat{p}}, b^{\infty}_{\hat{p}},  b^{i}_{\hat{q}}, b^{\infty}_{\hat{q}}$ analogously to \eqref{eq:defbip}-\eqref{eq:defbipinf}:
\begin{align}
b^{i}_{\hat{p}}(\cdot):= \d^{i}(\hat{p}, \cdot)-\d^{i}(\hat{p},x),  \quad  &\quad  b^{i}_{\hat{q}}(\cdot):= \d^{i}(\hat{q}, \cdot)-\d^{i}(\hat{q},x), \nonumber \\
b^{\infty}_{\hat{p}}(\cdot):= \lim_{t\to -\infty}  -t- \d_{E}(\hat{\ell}_{p}(t), \cdot) , \quad & \quad  b^{\infty}_{\hat{q}}(\cdot):= \lim_{t\to -\infty}  -t- \d_{E}(\hat{\ell}_{q}(t), \cdot). \nonumber
\end{align}
With analogous  arguments  as above we get
\begin{equation*}
b^{i}_{\hat{p}}\to  b^{\infty}_{\hat{p}}, \quad  b^{i}_{\hat{q}} \to  b^{\infty}_{\hat{q}} \quad \text{uniformly on bounded subsets  in the sense of  Definition \ref{def:pointConv}.}
\end{equation*} 
Since we are working in the euclidean space $\R^{k}$, it is not difficult to see that  $b^{\infty}_{p}+b^{\infty}_{\bar{p}}=0$ and  $b^{\infty}_{q}+b^{\infty}_{\bar{q}}=0$. Note that such equalities holds more generally in manifolds with non-negative Ricci curvature: the argument, used in the proof of the Cheeger-Gromoll Splitting Theorem, goes via maximum principle; here in any case one can argue more directly by using the geometry of the euclidean space. It follows that $ b^{i}_{p}+ b^{i}_{\hat{p}} \to 0$ and $ b^{i}_{p}+ b^{i}_{\hat{p}} \to 0$  uniformly on bounded sets in the sense of Definition \ref{def:pointConv}. Hence  there exists a function $\Phi(R|K,N)$ satisfying  $\lim_{R\to +\infty}\Phi(R|K,N)=0$ for fixed $K,N$, such that $0\leq { b}^i_{{p}}(y)+{ b}^i_{\hat{p}}(y)\leq \Phi(\frac1{r_i}|K,N)$ and $0\leq { b}^i_{{q}}(y)+{ b}^i_{\hat{q}}(y)\leq \Phi(\frac1{r_i}|K,N)$ for any $y\in  B_{1}^{\d^i}(x)$.

Using Proposition \ref{lemma-harmonicapprox},  for every $i \in \N$ we can construct harmonic approximations  ${\bf b}^{i}_{p}, {\bf b}^{i}_{q}$   of $b^{i}_{p}, b^{i}_{q}$, respectively, in  the unit ball $B_{1}^{\d^{i}}(x)$ of the space $(X,\d^{i},\mm^{x}_{r_{i}})$. 
Since $\d^{i}(p,x)=\frac{1}{r_{i}} \d(p,x)\to \infty$, $\d^{i}(q,x)=\frac{1}{r_{i}} \d(q,x)\to \infty$  and   the spaces $(X,\d^{i},\mm^{x}_{r_{i}})$ are RCD$^{*}(r_{i}^{2}K,N)$, so in particular RCD$^{*}(-1,N)$ for $i$ large enough, we infer that
\begin{align}
&\|{\bf b}^{i}_{p}- b^{i}_{p}\|_{L^\infty(B_1^{d_{i}}(x))}+ \frac{1}{\mm^{x}_{r_{i}}(B_{1}^{\d_{i}}(x))} \int_{B_{1}^{\d_{i}}(x)}|\D( {\bf b}^{i}_{p}-  b^{i}_{p})| ^{2} \, \d \mm^{x}_{r_{i}} \to 0 \quad  \text{as } i \to \infty, \label{eq:HarmApproxip}\\ 
&\|{\bf b}^{i}_{q}- b^{i}_{q}\|_{L^\infty(B_1^{d_{i}}(x))}+ \frac{1}{\mm^{x}_{r_{i}}(B_{1}^{\d_{i}}(x))} \int_{B_{1}^{\d_{i}}(x)}|\D( {\bf b}^{i}_{q}-  b^{i}_{q})| ^{2} \, \d \mm^{x}_{r_{i}} \to 0 \quad  \text{as } i \to \infty. \label{eq:HarmApproxiq}
\end{align}
The combination of \eqref{eq:bibinfL2}, \eqref{eq:HarmApproxip} and \eqref{eq:HarmApproxiq} yields
\begin{equation*}
{\bf b}^{i}_{p} \llcorner B_{1}^{\d_{i}}(x) \to  b^{\infty}_{p}  \llcorner B_{1}^{\d_{E}}(O), \quad {\bf b}^{i}_{q}  \llcorner B_{1}^{\d_{i}}(x) \to  b^{\infty}_{q}  \llcorner B_{1}^{\d_{E}}(O) \quad \text{strongly in $L^{2}$.} 
\end{equation*} 
Since by construction ${\bf \Delta} {\bf b}^{i}_{p} \llcorner B_{1}^{\d_{i}}(x)=0$, by Proposition  \ref{prop-1order-converge} we get that 
\begin{equation*}
{\bf b}^{i}_{p} \llcorner B_{1}^{\d_{i}}(x) \to  b^{\infty}_{p}  \llcorner B_{1}^{\d_{E}}(O), \quad {\bf b}^{i}_{q}  \llcorner B_{1}^{\d_{i}}(x) \to  b^{\infty}_{q}  \llcorner B_{1}^{\d_{E}}(O) \quad \text{strongly in $W^{1,2}$}. 
\end{equation*} 
But then the gradient estimates in \eqref{eq:HarmApproxip}-\eqref{eq:HarmApproxiq} give that
\begin{equation}\label{eq:bibinfW12}
{b}^{i}_{p} \llcorner B_{1}^{\d_{i}}(x) \to  b^{\infty}_{p}  \llcorner B_{1}^{\d_{E}}(O), \quad {b}^{i}_{q}  \llcorner B_{1}^{\d_{i}}(x) \to  b^{\infty}_{q}  \llcorner B_{1}^{\d_{E}}(O) \quad \text{strongly in $W^{1,2}$}. 
\end{equation} 
In particular, for every $\rho\in (0,1)$ we have
\begin{equation}\label{eq:intanglei}
\lim_{i \to \infty}  \frac{1}{ \mm^{x}_{r_{i}} (B^{\d_{i}}_{\rho}(x))} \int_{B^{\d_{i}}_{\rho}(x)} \la \nabla b^{i}_{p},  \nabla b^{i}_{q} \ra \, \d  \mm^{x}_{r_{i}}  = \frac{1}{{\mathcal L}_{k}(B^{\d_{E}}_{\rho}(O))} \int_{B^{\d_{E}}_{\rho}(O)} \la \nabla b^{\infty}_{p},  \nabla b^{\infty}_{q} \ra \, \d {\mathcal L}_{k}.
\end{equation}
We now analyze the two sides of  \eqref{eq:intanglei}. Recalling  \eqref{eq:LebsesqueAngle}, from the very definitions of $\mm^{x}_{r_{i}}$ and of $\d^{i}$ it follows that
\begin{align}
\lim_{i \to \infty}  \frac{1}{ \mm^{x}_{r_{i}} (B^{\d_{i}}_{ \rho}(x))} \int_{B^{\d_{i}}_{\rho}(x)} \la \nabla b^{i}_{p},  \nabla b^{i}_{q} \ra \, \d  \mm^{x}_{r_{i}}& =  \lim_{i \to \infty}   \frac{1}{ \mm(B^{\d}_{r_{i} \rho}(x))} \int_{B^{\d}_{r_{i}\rho}(x)} \la \nabla b_{p},  \nabla b_{q} \ra \, \d  \mm  \nonumber \\
 &=\angle p x q \label{eq:intangleii}.
\end{align}
On the other hand, since $b^{\infty}_{p}, b^{\infty}_{q}$ are the Busemann functions of the lines $\ell_{p}, \ell_{q}$ in $\R^{k}$ it is readily seen that  
\begin{equation} \label{eq:intangleeucl}
\lim_{\rho \downarrow 0} \frac{1}{{\mathcal L}_{k}(B^{\d_{E}}_{\rho}(O))} \int_{B^{\d_{E}}_{\rho}(O)} \la \nabla b^{\infty}_{p},  \nabla b^{\infty}_{q} \ra \, \d {\mathcal L}_{k} = \angle \bar{p} O \bar{q}.
\end{equation}
Putting together \eqref{eq:intanglei}, \eqref{eq:intangleii}  and  \eqref{eq:intangleeucl} finally yields
\begin{align*}
 \angle \bar{p} O \bar{q} &=   \lim_{\rho \downarrow 0} \frac{1}{{\mathcal L}_{k}(B^{\d_{E}}_{\rho}(O))} \int_{B^{\d_{E}}_{\rho}(O)} \la \nabla b^{\infty}_{p},  \nabla b^{\infty}_{q} \ra \, \d {\mathcal L}_{k} =  \lim_{\rho \downarrow 0}  \lim_{i \to \infty}  \frac{1}{ \mm^{x}_{r_{i}} (B^{\d_{i}}_{\rho}(x))} \int_{B^{\d_{i}}_{\rho}(x)} \la \nabla b^{i}_{p},  \nabla b^{i}_{q} \ra \, \d  \mm^{x}_{r_{i}}  \\
 &= \angle p x q,
\end{align*}
as desired.
\\

\textbf{Step 3}.  We claim that
\begin{equation}\label{eq:claimStep3}
 \angle p x q=\lim_{t \downarrow 0} \arccos \frac{2t^2-\d(\gamma^{xp}_{t}, \gamma^{xq}_{t})^{2}}{2t^{2}}.
\end{equation}
To this aim, first of all observe that the cosine formula  in $\R^{k}$  ensures that
\begin{equation}\label{eq:claimStep3Rk}
 \angle \bar{p} O \bar{q}= \arccos \frac{2-\d_{E}(\bar{p},  \bar{q})}{2}.
\end{equation}
Let now $t_{i} \downarrow 0$ be any  sequence and set $r_{i}:=t_{i}$. Define the rescaled spaces    $(X,  \d^{i}, x)$  as above, with $\d^{i}(\cdot,\cdot):=r_{i}^{-1} \d(\cdot,\cdot)$.
Notice first of all that the p-GH convergence of  $(X,  \d^{i}, x)$ to $(\R^{k}, \d_{E}, O)$ ensures that
\begin{equation}\label{eq:digammapq}
\lim_{i \to \infty} \d^{i}(\gamma^{xp}_{t_{i}}, \gamma^{xq}_{t_{i}})= \d_{E}(\bar{p}, \bar{q}).
\end{equation}
It follows that
\begin{align}
\lim_{i \to \infty} \frac{2t_{i}^2-\d(\gamma^{xp}_{t_{i}}, \gamma^{xq}_{t_{i}})^{2}}{2t_{i}^{2}}&=\lim_{i \to \infty} \frac{2-\d^{i}(\gamma^{xp}_{t_{i}}, \gamma^{xq}_{t_{i}})^{2}}{2} \overset{\eqref{eq:digammapq} }{=} \frac{2-\d_{E}(\bar{p}, \bar{q})}{2} \nonumber \overset{\eqref{eq:claimStep3Rk}}{=}\cos \left(\angle \bar{p} O \bar{q} \right) \nonumber  \\ 
&\overset{\eqref{eq:pfstep2}}{=}\cos \left(\angle pxq \right). \label{eq:Step3i} 
\end{align}
Since the sequence $t_{i}\downarrow 0$ was arbitrary, \eqref{eq:Step3i}  implies  \eqref{eq:claimStep3}.

The thesis then follows by combining \eqref{eq:pfstep2} and \eqref{eq:claimStep3}.

\end{proof}

\begin{remark}
The cosine formula  in $\R^{k}$  ensures that
\begin{equation}\label{eq:claimStep3Rk1}
 \angle \bar{p} O \bar{q}= \arccos \frac{s^2+t^2-\d_{E}(\ell_{p}(s), \ell_{q}(t))^{2}}{2st}  \quad \forall s,t >0.
\end{equation}
It is natural to ask if the same formula holds in the non-smooth case. This remains an open problem even for Ricci limit spaces, so a fortiori in $\rcdkn$ spaces.
\\ Here let us briefly mention that with analogous arguments as above one can show the weaker statement
\begin{equation}\label{eq:claimStep31}
 \angle p x q=\lim_{s,t \downarrow 0, \frac{1}{C} \leq \frac st \leq C} \arccos \frac{s^2+t^2-\d(\gamma^{xp}_{s}, \gamma^{xq}_{t})^{2}}{2st}, \quad \text{for every } C\geq 1.
\end{equation}

To this aim let   $s_{i}\downarrow 0, t_{i} \downarrow 0$ be any two sequences. Up to subsequences, we may assume that for all $i\in \N$ it holds either $0\leq s_{i} \leq t_{i}$ or   $0\leq t_{i} \leq s_{i}$. Without loss of generality we may assume the first case.  Up to further subsequences we may also assume that $s_{i}/t_{i}$ has a limit $\bar{s}\in (0,1]$ as $i\to \infty$.
Let  $r_{i}:=t_{i}\downarrow 0$ and define  the rescaled spaces  $(X,  \d^{i}, x)$  as above, with $\d^{i}(\cdot,\cdot):=r_{i}^{-1} \d(\cdot,\cdot)$.
Calling $s_{i}':=r_{i}^{-1} s_{i} \to \bar{s}$,  $t_{i}':=r_{i}^{-1} t_{i}=1$,  the p-GH convergence of  $(X,  \d^{i}, x)$ to $(\R^{k}, \d_{E}, O)$ ensures
\begin{equation*}
\lim_{i \to \infty} \d^{i}(\gamma^{xp}_{s_{i}}, \gamma^{xq}_{t_{i}})= \d_{E}(\ell_{p}(\bar{s}), \bar{q}).
\end{equation*}
Then we have
\begin{align}
\lim_{i \to \infty} \frac{s_{i}^2+t_{i}^2-\d(\gamma^{xp}_{s_{i}}, \gamma^{xq}_{t_{i}})^{2}}{2s_{i}t_{i}}&=\lim_{i \to \infty} \frac{(s'_{i})^2+(t'_{i})^2-\d^{i}(\gamma^{xp}_{s_{i}}, \gamma^{xq}_{t_{i}})^{2}}{2s'_{i}t'_{i}} \nonumber\\
&=  \lim_{i \to \infty} \frac{(s'_{i})^2+(t'_{i})^2-\d_{E}(\ell_{p}(s'_{i}), \ell_{q}(t_{i}'))^{2}}{2s'_{i}t'_{i}} \nonumber \\
&\overset{\eqref{eq:claimStep3Rk1}}{=}\cos \left(\angle \bar{p} O \bar{q} \right)  \nonumber \\
&\overset{\eqref{eq:pfstep2}}{=}\cos \left(\angle pxq \right). \label{eq:Step3i1} 
\end{align}
Since the sequences $s_{i},t_{i}\downarrow 0$ were arbitrary, \eqref{eq:Step3i1}  implies  \eqref{eq:claimStep31}.
\end{remark}

\def\cprime{$'$} \def\cprime{$'$}

\end{document}